\documentclass[leqno,10.5pt]{article} 
\usepackage{xcolor}  
\usepackage{amsmath, amssymb}  
\usepackage{amsthm} 
\usepackage{amscd}    
\usepackage{cancel}    
\usepackage{amsmath, amssymb}        
\usepackage{amsthm} 
\usepackage{ascmac}     
\usepackage{amscd}   
\usepackage{color} 
\usepackage[normalem]{ulem}

\newcommand{\ev}{{\rm Eval}}

\newcommand{\card}{{\rm Card}} 
  
\theoremstyle{plain} 
\newtheorem{theorem}{\indent\sc Theorem}[section]
\newtheorem{lemma}[theorem]{\indent\sc Lemma}

\newtheorem{proposition}[theorem]{\indent\sc Proposition}

\theoremstyle{definition} 
\newtheorem{definition}[theorem]{\indent\sc Definition}

\newtheorem{remark}[theorem]{\indent\sc Remark}

\newtheorem{notation}[theorem]{\indent\sc Notation}

\newcommand{\C}{\mathbb{C}} 
\newcommand{\R}{\mathbb{R}}
\newcommand{\ru}{{\R}}
\newcommand{\Q}{\mathbb{Q}}

\newcommand{\Z}{\mathbb{Z}} 
\newcommand{\qu}{{\Q}}

\newcommand{\N}{\mathbb{N}}

\newcommand{\const}{{rm\log(2)\rule{0mm}{4mm}+r\left(\log(rm+1)+rm\log\left(\rule{0mm}{3.5mm}\frac{rm+1}{rm}\right)\right)}}

\def\2{I\hspace{-.1em}I}

\title{Generalized hypergeometric $G$-functions \\ take linear independent values}
\author{\textsc{Sinnou David}, \textsc{Noriko Hirata-Kohno} and \textsc{Makoto Kawashima}}
\date{2022, March 1}
\begin{document}

\maketitle

%

\begin{abstract}
In this article, we show a new general linear independence criterion related to values of $G$-functions,
including the linear independence of values at algebraic points 
of contiguous hypergeometric functions, which is not known before.
Let $K$ be any algebraic number field and $v$ be a place of $K$.
Let $r\in\Z$ with $r\ge2$. Consider $a_1,\ldots,a_{r}, b_1,\ldots,b_{r-1}\in \Q\setminus\{0\}$ not being negative integers.
Assume neither $a_k$ nor $a_k+1-b_j$ be strictly positive integers $(1\le k \le r, 1\le j \le r-1)$.
Let $\alpha_1,\ldots,\alpha_m\in K\setminus\{0\}$ with $\alpha_1,\ldots,\alpha_m$ pairwise distinct.
By choosing sufficiently large $\beta\in \Z$ depending on $K$ and $v$ such that the points $\alpha_1/\beta,\ldots,\alpha_m/\beta$ are closed enough to the origin,
we prove that the $rm+1$ numbers~$:$
\begin{align*}
&{}_{r}F_{r-1} \biggl(\begin{matrix} a_1,\ldots, a_r\\ b_1, \ldots, b_{r-1} \end{matrix} \biggm| \dfrac{\alpha_i}{\beta}\biggr)\enspace, \ \ 
{}_{r}F_{r-1} \biggl(\begin{matrix} a_1+1,\ldots,\ldots,\ldots,a_r+1\\ b_1+1, \ldots, b_{r-s}+1,b_{r-s+1},\ldots,b_{r-1} \end{matrix} \biggm| \dfrac{\alpha_i}{\beta}\biggr)\enspace\\
&(1\le i \le m, 1\le s \le r-1)\end{align*}
and $1$ are linearly independent over $K$.
The essential ingredient is our \emph{term-wise formal} construction  of type II of Pad\'e approximants
together with new non-vanishing argument for the generalized Wronskian.
\end{abstract}


\maketitle

\section{Introduction}

{{The generalized hypergeometric $G$-function, in the sense of C.~L.~Siegel, is one of central objects from analytic point of view as well as number theoretical interest.
In the article, we study arithmetic properties of values of the generalized hypergeometric functions, relying on Pad\'e approximations of type II.
We provide a new general linear independence criterion for the values of the functions at several distinct points, over a given algebraic number field of any finite degree.
Our statement  extends previous ones due to D.~V.~Chudnovsky or D.~V.~Chudnovsky-G.~V.~Chudnovsky  in  \cite[Theorem 3.1]{ch2}
 \cite[Theorem I]{ch11},  \cite[Theorem 0.3]{ch12}  \cite[Theorem I]{Chubrothers} and Yu. Nesterenko \cite[Theorem 1]{Nest1} \cite[Theorem 1]{Nest}, which all
dealt with values at one point and over the rational number field.
We proceed constructions of Pad\'e approximants by our formal method, generalizing that used in \cite{DHK2,DHK3,DHK4}.
We are inspired,  together with those quoted above, by works due to 
A.~I.~Galochkin in  \cite{G1,G2}, V.~N.~Sorokin in \cite{Sorokin}, K.~V$\ddot{\text{a}}$$\ddot{\text{a}}$n$\ddot{\text{a}}$nen in  \cite{Va} 
and W. Zudi\-lin in  \cite{Z}, which gave several linear independence criteria, either over the field of rational numbers or quadratic imaginary fields, 
of values those concerns polylogarithmic function or
hypergeometric $G$-function. However, these previous results were either for
values at only one point, or in the case where the ground field was limited.
As related works, we refer to the algebraic independence announced in \cite[Theorem 3.4]{Chubrothers2}  of the two special values of
Gauss' hypergeometric functions
${}_{2}F_{1} \biggl(\begin{matrix} \dfrac{1}{2}, \dfrac{1}{2}\\ 1 \end{matrix} \biggm| \alpha\biggr)\enspace$ and
${}_{2}F_{1} \biggl(\begin{matrix} -\dfrac{1}{2}, \dfrac{1}{2}\\ 1 \end{matrix} \biggm|  \alpha\biggr)\enspace$ when $ \alpha$ is a non-zero algebraic number
supposed to be of small module, that later proved by Y.~Andr\'e in \cite{andre} with the $p$-adic analogue}.
We  also mention that the work by  F.~Beukers involves several algebraicity of values of the function \cite{beu2,B}.
A historical survey for further reference is given in \cite{DHK2, DHK3}, with comparaison which concerns earlier works. }

{This criterion indeed shows the linear independence of values of  generalized hypergeometric functions  including the contiguous ones,
whose functional linear independence
 has been discussed in  \cite{Nest1,Nest}.
Our contribution in the proof, if any, is an uncharted non-vanishing property for the generalized Wronskian of Hermite type, corresponding to the case of generalized hypergeometric $G$-function.}

\section{Notations and main result}
We collect some notations which we use throughout the article. 
Let $\Q$ be the rational number field and  $K$ be an algebraic number field of arbitrary degree $[K:\Q] < \infty$.
Let us denote by $\N$ the set of strictly positive integers.
We denote the set of places of $K$  by ${{\mathfrak{M}}}_K$ (by ${\mathfrak{M}}^{\infty}_K$ for infinite places, by ${{\mathfrak{M}}}^{f}_K$
for finite places, respectively).
For $v\in {{\mathfrak{M}}}_K$, we denote the completion of $K$ with respect to $v$ by $K_v$, and 
{{the completion of an algebraic closure of $K_v$ by  $\mathbb C_v$ (resp. for $v\in {\mathfrak{M}}^{\infty}_K$,
for $v\in{{\mathfrak{M}}}^{f}_K$)
}}.

Let $p, q\in \N$. Let $a_1,\ldots,a_{p}, b_1,\ldots,b_{q}\in \Q\setminus\{0\}$ be non-negative integer.
We define the generalized hypergeometric function by
\begin{eqnarray*}
{}_{p}F_{q} \biggl(\begin{matrix} a_1,\ldots, a_p\\ b_1, \ldots, b_{q} \end{matrix} \biggm| z\biggr)
=\displaystyle\sum_{k=0}^{\infty}\dfrac{(a_1)_k \cdots (a_{p})_k}{(b_1)_k\cdots(b_{q})_k}\dfrac{z^k}{k!}\enspace,
\end{eqnarray*}
where $(a)_k$ is the Pochhammer symbol: $(a)_0=1$, $(a)_k=a(a+1)\cdots (a+k-1)$.

For a rational number $x$, let us define $$\mu(x)={\displaystyle{\prod_{\substack{q:\text{prime} \\ q|{\rm{den}}(x)}}}}q^{q/(q-1)}\enspace.$$


Let us denote the normalized absolute value $| \cdot |_v$ for $v\in {{\mathfrak{M}}}_K$~:
\begin{align*}
&|p|_v=p^{-\tfrac{[K_v:\Q_p]}{[K:\Q]}} \ \text{if} \ v\in{{\mathfrak{M}}}^{f}_K \ \text{and} \ v\mid p\enspace,\\
&|x|_v=|\iota_v x|^{\tfrac{[K_v:\R]}{[K:\Q]}} \ \text{if} \ v\in {{\mathfrak{M}}}^{\infty}_K\enspace,
\end{align*}
where $p$ is a prime number and $\iota_v$ the embedding $K\hookrightarrow \C$ corresponding to $v$. 
{On $K_v^n$, the norm $\|\cdot\|_v$ denotes the norm of the supremum.}
Then we have the product formula 
\begin{align*} 
\prod_{v\in {{\mathfrak{M}}}_K} |\xi|_v=1 \ \ \text{for} \ \ \xi \in K\setminus\{0\}\enspace.
\end{align*}

\
 
Let $m$ be a positive integer and $\boldsymbol{\beta}:=(\beta_0,\ldots,\beta_m) \in K^{m+1} \setminus\{\bold{0}\}$.  
Define the absolute height of $\boldsymbol{\beta}$ by
\begin{align*}
&{\mathrm{H}}(\boldsymbol{\beta})=\prod_{v\in {{\mathfrak{M}}}_K} \max\{1, |\beta_0|_v,\ldots,|\beta_m|_v\}\enspace,
\end{align*}
{{and logarithmic absolute height by ${\rm{h}}(\boldsymbol{\beta})={\rm{log}}\, \mathrm{H}(\boldsymbol{\beta})$. 
Let $v\in \mathfrak{M}_K$. We denote ${\rm{log}}\max(1,|\beta_i|_v)$ by ${\mathrm{h}}_v(\boldsymbol{\beta})$. 
Then we have ${\mathrm{h}}(\boldsymbol{\beta})={\displaystyle{\sum_{v\in \mathfrak{M}_K}}}{\mathrm{h}}_v(\boldsymbol{\beta})$.
For a finite set $S\subset \overline{\Q}$, we define the denominator of $S$ by $${\rm{den}}(S)=\min \{1\le n\in \Z\mid n\alpha \ \text{are algebraic integer for all} \ \alpha \in S \}\enspace.$$}}

Let $m,r$ be strictly positive integers with $r\ge2$ and $\boldsymbol{\alpha}:=(\alpha_1,\ldots,\alpha_m)\in (K\setminus\{0\})^m$ whose coordinates are pairwise distinct. 
For $\beta\in K\setminus\{0\}$, define a real number 
\begin{align*}
V_v(\boldsymbol{\alpha},\beta) &=\displaystyle  \log\vert\beta\vert_{v_0}-rm{\mathrm{h}}(\boldsymbol{\alpha},\beta)-{{(rm+1)}}\log\, \|\boldsymbol{\alpha}\|_{v_0}+rm\log\, \|(\boldsymbol{\alpha},\beta)\|_{v_0}\\ 
&-\left(\const\right)\\
&-\sum_{j=1}^r\left(\log\mu(a_j)+2\log\mu(b_j)+\dfrac{{\rm{den}}(a_j){\rm{den}}(b_j)}{\varphi({\rm{den}}(a_j))\varphi({\rm{den}}(b_j))} \right)\enspace,
\end{align*}
where $\varphi$ is the Euler's totient function.

\

Now we are ready to state our main theorem.
\begin{theorem} \label{hypergeometric} 
Let $v_0$ be a place of $K$.
Let $a_1,\ldots,a_{r}, b_1,\ldots,b_{r-1}\in \Q\setminus\{0\}$ be non-negative integer.
Assume neither $a_k$ nor $a_k+1-b_j$ be strictly positive integers $(1\le k \le r, 1\le j \le r-1)$.
Suppose $V_{v_0}(\boldsymbol{\alpha},\beta)>0$. Then the $rm+1$ numbers~$:$
\begin{align*}
&{}_{r}F_{r-1} \biggl(\begin{matrix} a_1,\ldots, a_r\\ b_1, \ldots, b_{r-1} \end{matrix} \biggm| \dfrac{\alpha_i}{\beta}\biggr)\enspace, \ \ 
{}_{r}F_{r-1} \biggl(\begin{matrix} a_1+1,\ldots,\ldots,\ldots,a_r+1\\ b_1+1, \ldots, b_{r-s}+1,b_{r-s+1},\ldots,b_{r-1} \end{matrix} \biggm| \dfrac{\alpha_i}{\beta}\biggr)\enspace
\end{align*}
$(1\le i \le m, 1\le s \le r-1)$ and $1$ are linearly independent over $K$.
\end{theorem}

\

We mention that a linear independence criterion for values of generalized hypergeometric $G$-functions \emph{with cyclic coefficients} also follows from Theorem~$\ref{hypergeometric} $ based on
the same argument in \cite{DHK4}. We will join it in the context of our futur paper to avoid
heavy calculations in the current article.

This article is organized as follows.
In Section~$\ref{ghf}$, we describe our setup for generalized hypergeometric $G$-functions.
In Section~$\ref{pa}$, we proceed our construction of  Pad\'e approximants,  generalizing the method used in \cite{DHK2,DHK3,DHK4}.
Section~$\ref{nonvan}$ is devoted to show the non-vanishing property of the crucial determinant.
In Section~$\ref{proof}$, we give the proof of Theorem~$\ref{hypergeometric}$.
A more general statement, together with totally effective linear independence measures, is also given in this section by Theorem~$\ref{thm 2}$.

\section{Pad\'{e} approximation of generalized hypergeometric functions}\label{pa}
Throughout this section, denote by $K$ a field of characteristic $0$. For a variable $z$, we denote $z\tfrac{d}{dz}$ by $\theta_z$.
\subsection{Preliminaries}\label{ghf}
In this subsection, we introduce the generalized hypergeometric function. 
First let us introduce polynomials $A(X),B(X)\in K[X]$ satisfying $\max({\rm{deg}}\,A,{\rm{deg}}\,B)>0$. 
Assume 
\begin{align} \label{AB}
A(k)B(k)\neq 0 \ \ \ (k\ge0)\enspace.
\end{align} 
Notice that this assumption yields $A(\theta_t+k),B(\theta_t+k)\in {\rm{Aut}}_K(K[t])$ for any non-negative integer $k$.
Next, consider a sequence $\bold{c}:=(c_{k})_{k \ge0}$ satisfying $c_k\in K\setminus\{0\}$ and 
\begin{align}\label{recurrence 1} 
c_{k+1}=c_k\cdot\dfrac{A(k)}{B(k+1)} \ \ \ (k\ge0)\enspace.
\end{align}
We introduce the formal power series $$F(z):=\sum_{k=0}^{\infty}c_kz^{k+1}\enspace, $$
sometime also called generalized hypergeometric function.

By the recurrence relation $(\ref{recurrence 1})$, the series $F(1/z)$ is a solution of the differential equation:
\begin{align*}
\Big(B(-\theta_z)z-A(-\theta_z)\Big)f(z)=B(0)\enspace.
\end{align*} 
In order to construct Pad\'{e} approximants of the function $F(z)$, we introduce a power series, say,  contiguous to $F(z)$.

Put $r=\max({\rm{deg}}\,A,{\rm{deg}}\,B)$ and take $\gamma_1,\ldots,\gamma_{r-1}\in K$. Let $s$ be an integer with $0\le s \le r-1$. 
We define the power series $F_s(z)$ by 
\begin{align} \label{Fs1}
F_0(z)=F(z), \ \ \ F_{s}(z)=\sum_{k=0}^{\infty}(k+\gamma_1)\cdots (k+\gamma_s)c_kz^{k+1} \ \text{for} \ 1\le s \le r-1\enspace.
\end{align}
Notice that $F_s(1/z)$ satisfies
$$F_{s}(1/z)=(-\theta_z+\gamma_1-1)\cdots (-\theta_z+\gamma_s-1) (F_{0}(1/z))\enspace.$$
\begin{remark} \label{Main ex}
Let $p, q\in \N$, $a_1,\ldots,a_p, b_1,\ldots,b_q\in K\setminus\{0\}$ be non-negative integer.
Put $A(X)=(X+a_1+1)\cdots(X+a_p+1)$, $B(X)=(X+b_1)\cdots(X+b_q)(X+1)$ and define $$c_k=\dfrac{(a_1)_{k+1}\cdots(a_p)_{k+1}}{(b_1)_{k+1}\cdots(b_q)_{k+1}(k+1)!} \ \ (k\ge 0)\enspace.$$
Then $(c_k)_{k\ge0}$ satisfies $$c_{k+1}=c_k\cdot\dfrac{A(k)}{B(k+1)}\enspace.$$
For this sequence, we have $$F(1/z)={}_{p}F_{q} \biggl(\begin{matrix} a_1,\ldots, a_p\\ b_1, \ldots, b_{q} \end{matrix} \biggm| \dfrac{1}{z}\biggr)-1
\enspace.$$  
We assume $r:=p-1=q$. Put $\gamma_1=1,\gamma_2=b_{r-1},\ldots,\gamma_{r-1}=b_2$. Then the series $F_s(1/z)$ has the expression~:
\begin{align}
&F_0(1/z)={}_{r}F_{r-1} \biggl(\begin{matrix} a_1,\ldots, a_r\\ b_1, \ldots, b_{r-1} \end{matrix} \biggm| \dfrac{1}{z}\biggr)-1\enspace, \label{F0}\\
&F_s(1/z)
=\dfrac{a_1\cdots a_r}{b_1\cdots b_{r-s}}\cdot \dfrac{1}{z}\cdot {}_{r}F_{r-1} \biggl(\begin{matrix} a_1+1,\ldots,\ldots,\ldots,a_r+1\\ b_1+1, \ldots, b_{r-s}+1,b_{r-s+1},\ldots,b_{r-1} \end{matrix} \biggm| \dfrac{1}{z}\biggr)\enspace,
\label{Fs}
\end{align}
for $1\le s \le r-1$.
\end{remark}

\subsection{Construction of Pad\'e approximants}\label{pa}
Let $K$ be a field of characteristic $0$.  
We define the order function ${\rm{ord}}_{\infty}$ at ``$z=\infty$" by
\begin{align*} 
{\rm{ord}}_{\infty}:K((1/z))\rightarrow \Z\cup \{\infty\}; \ \ \sum_{k} \dfrac{{{c_k}}}{z^k}\mapsto \min\{k\in \Z\mid {{c_k}}\neq 0\}\enspace.
\end{align*}    
We first recall the following fact ({\it see} \cite{Feldman})~:
\begin{lemma} \label{pade}
Let $r$ be a positive integer, $f_1(z),\ldots,f_r(z)\in (1/z)\cdot K[[1/z]]$ and $\boldsymbol{n}:=(n_1,\ldots,n_r)\in \N^{r}$.
Put $N:=\sum_{i=1}^rn_i$.
Let $M$ be a positive integer  with $M\ge N$. Then there exists a family of polynomials 
$(P_0(z),P_{1}(z),\ldots,P_r(z))\in K[z]^{r+1}\setminus\{\bold{0}\}$ satisfying the following conditions~$:$
\begin{align*} 
&(i) \ {\rm{deg}}\,P_{0}(z)\le M\enspace,\\
&(ii) \ {\rm{ord}}_{\infty} (P_{0}(z)f_j(z)-P_j(z))\ge n_j+1 \ \text{for} \ 1\le j \le r\enspace.
\end{align*}
\end{lemma}

\begin{definition}
For the family of polynomials $(P_0(z),P_{1}(z),\ldots,P_r(z)) \in K[z]^{r+1}$ satisfying the properties $(i)$ and $(ii)$ of Lemma $\ref{pade}$, let us call
it, weight $\boldsymbol{n}$ and degree $M$ Pad\'{e} type approximants of $(f_1,\ldots,f_r)$.
For such $(P_0(z),P_{1}(z),\ldots,P_r(z))$, of $(f_1,\ldots,f_r)$, 
consider the family of formal Laurent series $(P_{0}(z)f_j(z)-P_{j}(z))_{1\le j \le r}$. We call it weight $\boldsymbol{n}$ degree $M$ Pad\'{e} type approximations of $(f_1,\ldots,f_r)$.
\end{definition}

\begin{notation} \label{notationderiprim}
\begin{itemize}
\item[$(i)$]  For $\alpha\in K$, denote by ${\ev}_{\alpha}$ the linear evaluation map $K[t]\longrightarrow K$, $P\longmapsto P(\alpha)$. Whenever there is an ambiguity in a setting of variables, we will denote the map by $\ev_{t\rightarrow \alpha}$.
\item[$(ii)$] For $P\in K[t]$, we denote by $[P]$ the multiplication by $P$ (the map $Q\longmapsto PQ$).
\item[$(iii)$] For  a $K$-automorphism $\varphi$ of a $K$-module $M$ and an integer  $k$, put
$$\varphi^{k}=\begin{cases}
\overbrace{\varphi\circ\cdots\circ\varphi}^{k-\text{times}} & \ \text{if} \ k>0\\
{\rm{id}}_M &  \ \text{if} \ k=0\\
\overbrace{\varphi^{-1}\circ\cdots\circ\varphi^{-1}}^{-k-\text{times}} & \ \text{if} \ k<0\enspace.
\end{cases}
$$
\end{itemize}
\end{notation}

\

Now we explicitly construct Pad\'{e} approximants of generalized hypergeometric functions at distinct points.
The following lemma is a key ingredient. 
\begin{lemma} \label{difference}
Let ${{k}}$ be a non-negative integer. 

$(i)$ Let $H(X)\in K[X]$. We have $[t^k]\circ H(\theta_t)=H(\theta_t-k)\circ [t^k].$

$(ii)$ Let $A,B\in K[X]$ be polynomials with  \eqref{AB}. Let $\bold{c}:=(c_k)_{k\ge0}$ be a sequence satisfying $c_k\in K\setminus\{0\}$ together with $(\ref{recurrence 1})$ for $A,B$.
Define $\mathcal{T}_{\bold{c}}\in {\rm{Aut}}_K(K[t])$ by 
\begin{align} \label{Phi}
\mathcal{T}_{\bold{c}}:K[t]\longrightarrow K[t]; \ t^k\mapsto \dfrac{t^k}{c_k}\enspace.
\end{align}
Then we have the relation, in the ring ${\rm{End}}_K(K[t])$,
$$[t^k]\circ \mathcal{T}_{\bold{c}}=\mathcal{T}_{\bold{c}}\circ A(\theta_t-1)\circ \cdots \circ A(\theta_t-k)\circ B(\theta_t)^{-1}\circ \cdots \circ B(\theta_t-k+1)^{-1}\circ [t^k] \enspace.$$
\end{lemma}
\begin{proof}
$(i)$ Let $n$ be a non-negative integer. 
We may assume $H(X)=X^n$. For any non-negative integer $m$, we have 
\begin{align} \label{LHS1}
[t^k]\circ \theta^{n}_t(t^m)=m^nt^{m+k}\enspace.
\end{align} 
On the other hand, we have 
\begin{align*} 
(\theta_t-k)^n\circ [t^k](t^m)=(k+m-k)^nt^{m+k}=m^nt^{m+k} \enspace.
\end{align*}
By $(\ref{LHS1})$ and the above identity, we obtain the assertion.

$(ii)$
Let $m$ be a non-negative integer. The recurrence relation $(\ref{recurrence 1})$ yields 
$$\dfrac{1}{c_{m+k}}=\dfrac{B(m+k)\cdots B(m+1)}{A(m+k-1)\cdots A(m)}\cdot\dfrac{1}{c_m} \enspace,$$
hence we obtain
\begin{align*}
[t^k]\circ \mathcal{T}_{\bold{c}}(t^m)&=\dfrac{t^{k+m}}{c_m}=\dfrac{1}{c_{m+k}}\dfrac{A(m+k-1)\cdots A(m)}{B(m+k)\cdots B(m+1)}t^{k+m}\\
&=\mathcal{T}_{\bold{c}}\circ A(\theta_t-1)\circ \cdots \circ A(\theta_t-k)\circ B(\theta_t)^{-1}\circ \cdots \circ B(\theta_t-k+1)^{-1}\circ [t^k](t^m) \enspace.
\end{align*}
which achieves the proof of $(ii)$.
\end{proof}
We are now ready for our construction of Pad\'{e} approximants, of the hypergeometric functions at distinct points.
Let $\bold{c}:=(c_k)_{k\ge0}$ be a sequence satisfying $c_k\in K\setminus\{0\}$ together with $(\ref{recurrence 1})$ for polynomials $A,B\in K[X]$.
Put $r=\max({\rm{deg}}\,A,{\rm{deg}}\,B)$. Let us fix $\gamma_1,\ldots,\gamma_{r-1}\in K$. 
We denote by $F_s(z)$ the power series defined in \eqref{Fs} for $\gamma_1,\ldots,\gamma_{r-1}\in K$.
Let $m$ be a strictly positive integer and $\alpha_1,\ldots,\alpha_m\in K\setminus\{0\}$ which are pairwise distinct.
For $0\leq s \leq r-1$, we shall introduce a $K$-homomorphism $\psi_{s,i}\in {\rm{Hom}}_K(K[t],K)$ by
$$\psi_{i,s}:K[t]\longrightarrow K; \ t^k\mapsto (k+\gamma_1)\cdots (k+\gamma_s)c_k\alpha^{k+1}_i \enspace,$$
where $(k+\gamma_1)\cdots (k+\gamma_s)=1$ for $s=0$ and $k\ge0$.
\begin{proposition} \label{GHG pade} $(${\it confer} {\rm{\cite[Theorem $5.5$]{ch brothers}}} $)$
We use the notations as above.
For a non-negative integer $\ell$, we define polynomials$:$
\begin{align}
&P_{\ell}(z)=\left[\dfrac{1}{(n-1)!^r}\right]\circ {\rm{Eval}}_z\circ \mathcal{T}_{\bold{c}} \bigcirc_{j=1}^{n-1} B(\theta_t+j)\left(t^{\ell}\prod_{i=1}^m(t-\alpha_i)^{rn}\right)\enspace, \label{Pl}\\
&P_{\ell,i,s}(z)=\psi_{i,s}\left(\dfrac{P_{\ell}(z)-P_{\ell}(t)}{z-t}\right) \ \text{for} \ 1\le i \le m, 0\le s \le r-1\enspace, \label{Plis}
\end{align}
where $\mathcal{T}_{\boldsymbol{c}}\in {\rm{Aut}}_K(K[t])$ defined in $(\ref{Phi})$.
Then $(P_{\ell}(z),P_{\ell,i,s}(z))_{1\le i \le m, 0\le s \le r-1}$ forms a weight $(n,\ldots,n)\in \N^{rm}$ and degree $rmn+\ell$ Pad\'{e} type approximants of
$(F_{s}(\alpha_i/z))_{1\le i \le m, 0\le s \le r-1}$.
\end{proposition}
\begin{proof}
By the definition of $P_{\ell}(z)$, we have
$$
{\rm{deg}}\,P_{\ell}(z)=rmn+\ell\enspace.
$$
Hence the required condition on the degree is verified. 
By the definition of $\mathcal{T}_{\bold{c}}$ and $\psi_{i,s}$, we have 
\begin{align}
&\psi_{{i,s}}=\psi_{i,0}\circ (\theta_t+\gamma_1)\circ \cdots \circ (\theta_t+\gamma_s)\enspace, \label{phi 0}\\
&\psi_{i,0}\circ \mathcal{T}_{\bold{c}}=[\alpha_i]\circ {\rm{Eval}}_{\alpha_i} \ \text{for} \ 1\le i \le m\enspace. \label{eval i}
\end{align}
Put $R_{\ell,i,s}(z)=P_{\ell}(z)F_{s}(\alpha_i/z)-P_{\ell,i,s}(z)$.
Then, by the definition of $R_{\ell,i,s}(z)$, we obtain
\begin{align*}
R_{\ell,i,s}(z)&=\displaystyle P_{\ell}(z)\psi_{i,s}\left(\dfrac{1}{z-t}\right)-P_{\ell,i,s}(z)
                                                       =\displaystyle\psi_{i,s}\left(\dfrac{P_{\ell}(t)}{z-t}\right)
                                                       =\sum_{k=0}^{\infty}\dfrac{\psi_{i,s}(t^kP_{\ell}(t))}{z^{k+1}}\enspace.
\end{align*}
Let $k$ be an integer with $0\le k \le n-1$. 
By Lemma $\ref{difference}$ $(i)$  $(ii)$, we have 
\begin{align*}
(n-1)!^rt^kP_{\ell}(t)&=[t^k]\circ {\mathcal{T}}_{\bold{c}} \bigcirc_{j=1}^{n-1}B(\theta_t+j)\left(t^{\ell}\prod_{i=1}^m(t-\alpha_i)^{rn}\right) \\
&={\mathcal{T}}_{\bold{c}} \bigcirc_{j^{\prime}=1}^kA(\theta_t-j^{\prime})  \bigcirc_{j^{\prime \prime}=0}^{k-1}B(\theta_t-j^{\prime \prime})^{-1} \bigcirc_{j=1}^{n-1}B(\theta_t+j-k)\left(t^{\ell+k}\prod_{i=1}^m(t-\alpha_i)^{rn}\right)\\
&={\mathcal{T}}_{\bold{c}}\bigcirc_{j^{\prime}=1}^kA(\theta_t-j^{\prime})\bigcirc_{j=1}^{n-1-k}B(\theta_t+j)\left(t^{\ell+k}\prod_{i=1}^m(t-\alpha_i)^{rn}\right)\enspace,
\end{align*}
where $\bigcirc_{j^{\prime}=1}^kA(\theta_t-j^{\prime})={\rm{id}}_{K[t]}$ if $k=0$. 
Therefore we have 
{\small{
\begin{align}
\psi_{i,s}((n-1)!^rt^kP_{\ell}(t))
&=\psi_{i,s}\circ {\mathcal{T}}_{\bold{c}}\bigcirc_{j^{\prime}=1}^kA(\theta_t-j^{\prime})\bigcirc_{j=1}^{n-1-k}B(\theta_t+j)\left(t^{\ell+k}\prod_{i=1}^m(t-\alpha_i)^{rn}\right) \nonumber\\
&=\psi_{i,0}\circ {\mathcal{T}}_{\bold{c}}\bigcirc_{j^{\prime \prime}=1}^s (\theta_t+\gamma_{j^{\prime \prime}})\bigcirc_{j^{\prime}=1}^kA(\theta_t-j^{\prime})\bigcirc_{j=1}^{n-1-k}B(\theta_t+j)\left(t^{\ell+k}\prod_{i=1}^m(t-\alpha_i)^{rn}\right) \label{1}\\
&=[\alpha_i]\circ {\rm{Eval}}_{\alpha_i}\bigcirc_{j^{\prime \prime}=1}^s (\theta_t+\gamma_{j^{\prime \prime}})\bigcirc_{j^{\prime}=1}^kA(\theta_t-j^{\prime})\bigcirc_{j=1}^{n-1-k}B(\theta_t+j)\left(t^{\ell+k}\prod_{i=1}^m(t-\alpha_i)^{rn}\right)\enspace. \label{2}
\end{align} 
}}
Note that, in $(\ref{1})$, $(\ref{2})$, we use $(\ref{phi 0})$ and $(\ref{eval i})$ respectively.
Since we have 
$${\rm{deg}} \left(\prod_{j^{\prime \prime}=1}^s (X+\gamma_{j^{\prime \prime}})\prod_{j^{\prime}=1}^kA(X-j^{\prime})\prod_{j=1}^{n-1-k}B(X+j)\right)
\le s+rk+r(n-1-k)\le rn-1\enspace$$
thanks to the Leibniz rule, the polynomial $\bigcirc_{j^{\prime \prime}=1}^s (\theta_t+\gamma_{j^{\prime \prime}})\bigcirc_{j^{\prime}=1}^kA(\theta_t-j^{\prime})\bigcirc_{j=1}^{n-1-k}B(\theta_t+j)\left(t^{\ell+k}\prod_{i=1}^m(t-\alpha_i)^{rn}\right)$ is contained in the ideal $(t-\alpha_i)={\rm{ker}}\, {\rm{Eval}}_{\alpha_i}$. 
Consequently we have
$$\psi_{i,s}(t^kP_{\ell}(t))=0 \ \text{for} \ 0\le k \le n-1, 1\le i \le m, 0\le s \le r-1\enspace.$$
By the above expansion of $R_{\ell,i,s}(z)$, we obtain
$$
{\rm{ord}}_{\infty}R_{\ell,i,s}(z)\ge n+1 \ \text{for} \  1\le i \le m, 0\le s \le r-1\enspace,
$$
hence Proposition $\ref{GHG pade}$ follows.
\end{proof}

{{We should mention that this construction was also considered by D.~V.~Chudnovsky and G.~V.~Chudnovsky in \cite[Theorem $5.5$]{ch brothers},
but without arithmetic application. See also a related work in \cite{Matala}.}}

\begin{remark}
The polynomial $P_{\ell}(z)$ does not depend on the choice of $\gamma_1,\ldots,\gamma_{r-1}\in K$.
By contrast, the polynomials $P_{\ell,i,s}(z)$ depend on these choice.
\end{remark}
\begin{remark}
Let $r,m$ be strictly positive integers. Let $x\in K$, {{supposed to be non-negative integer}} and $\alpha_1,\ldots,\alpha_m\in K\setminus\{0\}$ be pairwise distinct.
Put $A(X)=B(X)=(X+x+1)^r$ and $c_k=1/(k+x+1)^r$. Then we have $$c_{k+1}=c_k\cdot\dfrac{A(k)}{B(k+1)}\enspace.$$
Put $\gamma_1=\cdots=\gamma_{r-1}=x+1$. This gives us 
\begin{align} \label{Fis}
F_{s}(\alpha_i/z)=\sum_{k=0}^{\infty}\dfrac{1}{(k+x+1)^{r-s}}\cdot\dfrac{\alpha^{k+1}_i}{z^{k+1}}=\Phi_{r-s}(x,\alpha_i/z) \ \ (1\le i \le m, 0\le s \le r-1)\enspace,
\end{align}
where $\Phi_{s}(x,1/z)$ is the $s$-th Lerch function (generalized polylogarithmic function; {\it confer} \cite{DHK3}). 
In this case, we have ${\mathcal{T}}_{\bold{c}}=\dfrac{(\theta_t+x+1)^r}{(x+1)^r}$ and 
$$P_{\ell}(z)=\left[\dfrac{1}{(x+1)^r \cdot (n-1)!^r}\right]\circ {\rm{Eval}}_z \bigcirc_{j=1}^n(\theta_t+x+j)^r\left(t^{\ell}\prod_{i=1}^m(t-\alpha_i)^{rn}\right)\enspace.$$
The polynomial $\tfrac{(x+1)^r}{n^r}P_{\ell}(z)$ gives Pad\'{e} type approximants of this Lerch functions in \cite[Theorem $3.8$]{DHK2}.
\end{remark}

\section{Non-vanishing of the generalized Wronskian of Hermite type}\label{nonvan} 
Let $K$ be a field of characteristic $0$ and $A(X),B(X)\in K[X]$ satisfying $\eqref{AB}$.
From this section to last, we assume ${\rm{deg}}\,A={\rm{deg}}\,B>0$ and put ${\rm{deg}}\, A=r$. 
We shall choose a sequence $\boldsymbol{c}:=(c_k)_{k\ge0}$ satisfying $c_k\in K\setminus\{0\}$ and \eqref{recurrence 1} for the given polynomials $A(X),B(X)$.
Let $\boldsymbol{\alpha}:=(\alpha_1,\ldots,\alpha_m)\in (K\setminus\{0\})^m$ whose coordinates are pairwise distinct and $\gamma_1,\ldots,\gamma_{r-1}\in K$.
Let us fix a positive integer $n$. 
For a non-negative integer $\ell$ with $0\le \ell \le rm$, recall the polynomials $P_{\ell}(z), P_{\ell,i,s}(z)$ defined in $(\ref{Pl})$ and $(\ref{Plis})$.
We define column vectors $\vec{p}_{\ell}(z)\in K[z]^{rm+1}$ by
\begin{align*}
&\vec{p}_{\ell}(z)={}^t\Biggl(P_{\ell}(z),
{P_{\ell,1,r-1}(z),\ldots, P_{\ell,1,0}(z)}, \ldots, {P_{\ell,m,r-1}(z),\ldots, P_{\ell,m,0}(z)}\Biggr)\enspace,
\end{align*}
and put
$$
\Delta_n(z)=\Delta(z)=
{\rm{det}}  {\begin{pmatrix}
\vec{p}_{0}(z) \ \cdots \ \vec{p}_{rm}(z)
\end{pmatrix}}\enspace.
$$

\

The aim of this section is to prove the following proposition.
\begin{proposition} \label{non zero det}
The determinant $\Delta(z)$ satisfies $\Delta(z)\in K\setminus\{0\}. $
\end{proposition}

\subsection{First Step}
\begin{lemma} \label{sufficient condition}
We have $\Delta(z)\in K$.
\end{lemma}
\begin{proof}
We denote the remainder function $R_{\ell,i,s}(z):=P_{\ell}(z)F_{s}(\alpha_i/z)-P_{\ell,i,s}(z)$ ($0\le \ell \le rm$, $1\le i \le m$, $0\le s \le r-1$).
For the matrix in $\Delta(z)$, multiplying the 
first row by the $F_{s}(\alpha_i/z)$ and adding it to the $(i-1)r+s+1$-th row ($1\le i \le m$, $0\le s \le r-1$), we obtain
                     $$ 
                     \Delta(z)=(-1)^{rm}{\rm{det}}
                     {\begin{pmatrix}
                     P_{0}(z) & \dots &P_{rm}(z)\\
                     R_{0,1,r-1}(z) & \dots & R_{rm,1,r-1}(z)\\
                     \vdots    & \ddots & \vdots  \\
                     R_{0,1,0}(z) & \dots & R_{rm,1,0}(z)\\
                     \vdots & \ddots & \vdots\\
                     R_{0,m,r-1}(z) & \dots & R_{rm,m,r-1}(z)\\
                     \vdots    & \ddots & \vdots  \\
                     R_{0,m,0}(z) & \dots & R_{rm,m,0}(z)\\
                     \end{pmatrix}}\enspace. 
                     $$ 
We denote the $(s,t)$-th cofactor by $\Delta_{s,t}(z)$  of the matrix in the right hand side above.
Then we have, by developing along the first row~:
\begin{align} \label{formal power series rep delta}
\Delta(z)=(-1)^{rm}\left(\sum_{\ell=0}^{rm}P_{\ell}(z)\Delta_{1,\ell+1}(z)\right)\enspace.
\end{align} 
Since we have  
\begin{align*}
{\rm{ord}}_{\infty} R_{\ell,i,s}(z)\ge n+1 \ \text{for} \ 0\le \ell \le rm, \ 1\le i\le m \ \text{and} \ 0\le s \le r-1\enspace,
\end{align*}
we obtain
$$
{\rm{ord}}_{\infty}\Delta_{1,\ell+1}(z)\ge (n+1)rm\enspace.
$$
The fact  ${\rm{deg}}\,P_{\ell}(z)=rmn+\ell$ with the lower bound above yields
$$
P_{\ell}(z)\Delta_{1,\ell+1}(z)\in (1/z)\cdot K[[1/z]] \ \text{for} \ 0\le \ell \le rm-1\enspace,
$$
and
\begin{align} \label{const+laurent series}
P_{rm}(z)\Delta_{1,rm+1}(z)\in K[[1/z]]\enspace.
\end{align}
In the relation above, the constant term of $P_{rm}(z)\Delta_{1,rm+1}(z)$ equals to~:
$$
\text{Coefficient of} \ z^{rm(n+1)} \ \text{of} \ P_{rm}(z) \times~~\text{Coefficient of} \ 1/z^{rm(n+1)} \ \text{of} \ \Delta_{1,rm+1}(z)\enspace.
$$ 
thanks to the fact that $\Delta(z)$ in $(\ref{formal power series rep delta})$ is a polynomial of non-positive valuation in $z$ with respect to ${\rm{ord}}_{\infty}$, it is necessarily to be a constant. 
Moreover, the terms of strictly negative valuation  in $z$, they  have to cancel out, hence 
\begin{align}\label{in K}
\Delta(z)=(-1)^{rm}\cdot \left(\sum_{\ell=0}^{rm}P_{\ell}(z)\Delta_{1,\ell+1}(z)\right)=(-1)^{rm}\times \text{Constant term of} \ P_{rm}(z)\Delta_{1,rm+1}(z)\in K\enspace.
\end{align}
This completes the proof of Lemma $\ref{sufficient condition}$.
\end{proof}

\subsection{Second step} 
We now start the second procedure, by factoring $\Delta$ as an element of $K(\alpha_1,\ldots,\alpha_m)$.
We use the same notations as in the proof of Lemma $\ref{sufficient condition}$.
By the equalities $(\ref{const+laurent series})$ and $(\ref{in K})$, we have
{\small{\begin{align} \label{in K2}
\Delta(z)=(-1)^{rm}\times\text{Coefficient of} \ z^{rm(n+1)} \ \text{of} \ P_{rm}(z)\times\text{Coefficient of} \ 1/z^{rm(n+1)} \ \text{of} \ \Delta_{1,rm+1}(z)\enspace.
\end{align}}}
Define a column vector ${\vec{q}}_{\ell}\in K^{rm}$ by
\begin{align*}
\vec{q}_{\ell}={}^t\Biggl(\psi_{{1,r-1}}(t^nP_{\ell}(t)),\ldots, \psi_{{1,0}}(t^nP_{\ell}(t)),\ldots, \psi_{{m,r-1}}(t^nP_{\ell}(t)),\ldots, \psi_{{m,0}}(t^nP_{\ell}(t))\Biggr)\enspace.
\end{align*}
By the definition of $\Delta_{n,1,rm+1}(z)$ with the identities
\begin{align*}
R_{\ell,i,s}(z)=\sum_{k=n}^{\infty}\dfrac{\psi_{{i,s}}(t^kP_{\ell}(t))}{z^{k+1}} \ \text{for} \ 0\le \ell \le rm, \ 1\le i \le m \ \text{and} \ 0\le s \le r-1\enspace,
\end{align*} 
we have 
$$
\text{Coefficient of} \ 1/z^{rm(n+1)} \ \text{of} \ \Delta_{n,1,rm+1}(z)={\rm{det}}
{\begin{pmatrix}
\vec{q}_{0} \ \cdots \ \vec{q}_{rm-1}
\end{pmatrix}}\enspace.
$$
By  $(\ref{in K2})$ with the above identity, we have 
\begin{align} \label{bunkai 0} 
\Delta(z)=(-1)^{rm}\dfrac{1}{(rmn+rm)!}\left(\dfrac{d}{dz}\right)^{rmn+rm}P_{m,rm}(z)\cdot
{\rm{det}}{\begin{pmatrix}
\vec{q}_{0} \ \cdots \ \vec{q}_{rm-1}
\end{pmatrix}}\enspace. 
\end{align} 
Note that, by the definition of $P_{rm}(z)$, we have ${\rm{deg}}\,P_{rm}=(n+1)rm$ and thus 
$$\dfrac{1}{(rmn+rm)!}\left(\dfrac{d}{dz}\right)^{rmn+rm}P_{m,rm}(z)\neq 0\enspace.$$

\subsection{Third step}
Relying on $(\ref{bunkai 0})$, we study here the values
\begin{align} \label{Theta}
\Theta= 
{\rm{det}}  {\begin{pmatrix}
\vec{q}_{0} \ \cdots \ \vec{q}_{rm-1}
\end{pmatrix}}\enspace.
\end{align}
From this subsection, we specify the choice of $\gamma_1,\ldots,\gamma_{r-1}\in K$ as follows.
Replacing $K$ by an appropriate finite extension, we may assume $A(X),B(X)$ be decomposable in $K$. 
Put $$A(X)=(X+\eta_1)\cdots (X+\eta_r), \ B(X)=(X+\zeta_1)\cdots (X+\zeta_r)\enspace,$$ 
where $\eta_1,\ldots, \eta_r,\zeta_1,\ldots,\zeta_r\in K\setminus\{0\}$, {{ being non-negative integer}}.
Take a sequence $(\gamma_i)_{1\le i \le {{r(n+1)-1}}}$ of $K$ with $\gamma_1=\zeta_r,\ldots, \gamma_r=\zeta_1$.
For each $0\le s \le r-1$, there exists a sequence $(a_{k,s})_{0\le k \le rn}\in K^{rn+1}$ with 
\begin{align}\label{expansion}
\prod_{j=1}^nA(X-j)=\sum_{k=0}^{rn}a_{k,s}\prod_{{{w}}=1}^k(X+\gamma_{r-s-1+{{w}}})\enspace,
\end{align}
where it read $\prod_{{{w}}=1}^k(X+\gamma_{r-s-1+{{w}}})=1$ if $k=0$. 
We now simplify the determinant $\Theta$ using the quantities $a_{0,s}$ to prove the non-vanishing property of $\Theta$.
\begin{lemma} \label{simplify}
Put $H_{\ell}(t)=t^{\ell}\prod_{i=1}^m(t-\alpha_i)^{rn}$ for $0\le \ell \le rm-1$.
Then we have
\begin{align*}
\Theta=\dfrac{\prod_{i=1}^m\alpha^r_i\prod_{s=0}^{r-1}a_{0,s}^m}{(n-1)!^{r^2m}}\cdot {\rm{det}}\left({\rm{Eval}}_{\alpha_i} \bigcirc_{{{w}}=0}^s(\theta_t+\gamma_{r-s+{{w}}})^{-1} (t^nH_{\ell}(t))\right)_{\substack{0\le \ell \le rm-1 \\ 1\le i \le m, 0\le s \le r-1}}.
\end{align*}
\end{lemma}
\begin{proof}
Using \eqref{phi 0} and \eqref{expansion}, we have 
{\small{\begin{align}
\psi_{{i,r-1-s}}\circ {\mathcal{T}}_{\bold{c}}\bigcirc_{j=1}^n A(\theta_t-j)\circ B(\theta_t)^{-1}&=\sum_{k=0}^sa_{k,s}\psi_{{i,0}}\circ {\mathcal{T}}_{\bold{c}} \bigcirc_{{{w}}=1}^{r-s-1+k}(\theta_t+\gamma_{{{w}}})\circ B(\theta_t)^{-1} \label{kichaku 1}\\
&+\sum_{k=s+1}^{rn}a_{k,s}\psi_{{i,0}} \circ {\mathcal{T}}_{\bold{c}}\circ B(\theta_t) \bigcirc_{{{w}}=r+1}^{r-s-1+k}(\theta_t+\gamma_{{{w}}})\circ B(\theta_t)^{-1}. \nonumber
\end{align}}}
Since $${\rm{deg}}\prod_{{{w}}=r+1}^{r-s-1+k}(X+\gamma_{{{w}}})=k-s-1\le rn-1$$ ($s+1 \le k \le rn$), by the Leibniz rule, the polynomial
$\bigcirc_{{{w}}=r+1}^{r-s-1+k}(\theta_t+\gamma_{{w}}) (t^nH_{\ell}(t))$ belongs to the ideal $(t-\alpha_i)={\rm{ker}}\, {\rm{Eval}}_{\alpha_i}$. 
Therefore, using \eqref{eval i}, we obtain
\begin{align*} 
&\sum_{k=s+1}^{rn}a_{k,s}\psi_{{i,0}} \circ {\mathcal{T}}_{\bold{c}}\circ B(\theta_t) \bigcirc_{{{w}}=r+1}^{r-s-1+k}(\theta_t+\gamma_{{w}}) \circ B(\theta_t)^{-1}(t^nH_{\ell}(t))\nonumber\\
&=\sum_{k=s+1}^{rn}a_{k,s}[\alpha_i]\circ {\rm{Eval}}_{\alpha_i} \bigcirc_{{{w}}=r+1}^{r-s-1+k}(\theta_t+\gamma_{{{w}}}) (t^nH_{\ell}(t))=0\enspace.  
\end{align*}
By the above equality with $(\ref{kichaku 1})$, we have 
\begin{align*}
\psi_{{i,r-1-s}}\circ {\mathcal{T}}_{\bold{c}}\bigcirc_{j=1}^n A(\theta_t-j)\circ B(\theta_t)^{-1}(t^nH_{\ell}(t))
&=\sum_{k=0}^sa_{k,s}\psi_{{i,0}}\circ {\mathcal{T}}_{\bold{c}} \bigcirc_{{{w}}=1}^{r-s-1+k}(\theta_t+\gamma_{{w}})\circ B(\theta_t)^{-1} (t^nH_{\ell}(t))\\
&=\sum_{k=0}^sa_{k,s}[\alpha_i]\circ {\rm{Eval}}_{\alpha_i} \bigcirc_{{{w}}=k}^{s}(\theta_t+\gamma_{r-s+{{w}}})^{-1}(t^nH_{\ell}(t))\enspace.
\end{align*}
Interpreting the relations above as linear manipulations of lines, the columns let the determinant unchanged. This completes the proof of Lemma $\ref{simplify}$.
\end{proof}
We now study when the quantity $\prod_{s=0}^{r-1}a_{0,s}^m$ does not vanish. 
The following lemma will be used to calculate each $a_{0,s}$.
\begin{lemma} \label{constant term}
Let $u$ be a strictly positive integer and $\tilde{\gamma}_1,\ldots,\tilde{\gamma}_u,\tilde{\eta}_1,\ldots,\tilde{\eta}_u\in K$.
Denote $$(X+\tilde{\eta}_1)\cdots(X+\tilde{\eta}_u)=b_{u,0}+\sum_{k=1}^ub_{u,k}(X+\tilde{\gamma}_1)\cdots(X+\tilde{\gamma}_k)\enspace,$$
with $b_{u,0},b_{u,1},\ldots,b_{u,u}\in K$. Then we have $b_{u,0}=(\tilde{\eta}_1-\tilde{\gamma}_1)\cdots(\tilde{\eta}_u-\tilde{\gamma}_1)$. 
\end{lemma}
\begin{proof}
We prove the lemma by induction on $u$. 
In the case of $u=1$, we have $$X+\tilde{\eta}_1=\tilde{\eta}_1-\tilde{\gamma}_1+(X+\tilde{\gamma}_1)\enspace.$$
This shows $b_{1,0}=\tilde{\eta}_1-\tilde{\gamma}_1$ then yields the assertion.
Suppose that  the current lemma be true for $u\ge 1$. We show its validity for $u+1$. In this case we get
\begin{align*}
&(X+\tilde{\eta}_1)\cdots(X+\tilde{\eta}_u)(X+\tilde{\eta}_{u+1})=\left[b_{u,0}+\sum_{k=1}^ub_{u,k}(X+\tilde{\gamma}_1)\cdots(X+\tilde{\gamma}_k)\right](X+\tilde{\eta}_{u+1})\\
&=b_{u,0}(X+\tilde{\gamma}_1+\tilde{\eta}_{u+1}-\tilde{\gamma}_1)+\sum_{k=1}^ub_{u,k}(X+\tilde{\gamma}_1)\cdots(X+\tilde{\gamma}_k)(X+\tilde{\gamma}_{k+1}+\tilde{\eta}_{u+1}-\tilde{\gamma}_{k+1})\enspace.
\end{align*}
The above identity yields $b_{u+1,0}=b_{u,0}(\tilde{\eta}_{u+1}-\tilde{\gamma}_1)$. By induction hypothesis for $b_{u,0}$, we conclude 
$$b_{u+1,0}=(\tilde{\eta}_1-\tilde{\gamma}_1)\cdots(\tilde{\eta}_u-\tilde{\gamma}_1)(\tilde{\eta}_{u+1}-\tilde{\gamma}_1)\enspace.$$
This completes the proof of Lemma $\ref{constant term}$.
\end{proof}
\begin{proposition}
The following two properties are equivalent.

$(i)$ The value $\prod_{s=0}^{r-1}a_{0,s}^m$ is non-zero.

$(ii)$ For $1\le i ,j\le r$ and $1\le k \le n$, we have $\eta_i-k-\zeta_j\neq 0.$  
\end{proposition}
\begin{proof}
Let $s$ be an integer with $0\le s \le r-1$.
Applying Lemma $\ref{constant term}$ with $u=rn$ and
$$(\tilde{\gamma}_1,\ldots,\tilde{\gamma}_{rn})=(\gamma_{r-s},\ldots,\gamma_{r(n+1)-s-1}) , \ \ \ (\tilde{\eta}_1,\ldots,\tilde{\eta}_{rn})=(\eta_i-k)_{1\le i \le r, 1\le k \le n}\enspace,$$
we get $$a_{0,s}=\prod_{i=1}^r \left[\prod_{k=1}^{n}(\eta_i-k-\gamma_{r-s})\right], \  \ (0\le s \le r-1)\enspace.$$
Since $\gamma_{r-s}=\zeta_{s+1}$ for $0\le s \le r-1$, the proposition follows.
\end{proof}
In the following, we assume $$\eta_i-\zeta_j \ \text{{\emph{not be}} strictly positive integers for} \ 1\le i,j \le r\enspace.$$
\subsection{Fourth step}
Now, we take the ring $K[t_{i,s}]_{1\leq i\leq m,0\leq s\leq r-1}$, the ring of polynomials in $rm$ variables over $K$.  
Recall the polynomial $B(X)$ is decomposed as $B(X)=(X+\zeta_1)\cdots(X+\zeta_r)$ with $\zeta_i\in K$ which are not negative integer. 
We choose $(\gamma_i)_{1\le i \le r}\in K^r$ by $\gamma_1=\zeta_r,\ldots,\gamma_r=\zeta_1$.
For each variable $t_{i,s}$, one has a well defined map for $\alpha\in K$~:
{\small
{\begin{align}
\label{defvarphi}
{{\tilde{\psi}}}_{\alpha,i,s}={\rm{Eval}}_{t_{i,s}\rightarrow \alpha} \bigcirc_{{{w}}=0}^{s}(\theta_{t_{i,s}}+\gamma_{r-s+{{w}}})^{-1}:K[t_{i,s}]_{\substack{1\leq i\leq m \\ 0\leq s\leq r-1}}\longrightarrow K[t_{i',s'}]_{(i',s')\neq (i,s)}; \ \
t^k_{i,s} \mapsto \dfrac{\alpha^k}{\prod_{{{w}}=0}^s(k+\gamma_{r-s+{{w}}})}\enspace. 
\end{align}}}
Using the definition above where $K[t_{i,s}]_{1\leq i\leq m,0\leq s\leq r-1}$ is seen as the  one variable polynomial ring $K'[t_{i,s}]$ over $K'=K[t_{i',s'}]_{(i',s')\neq (i,s)}$.
  
We now define for non-negative integers $n,u$
$$
{\hat{P}}_{n,u}(t_{i,s})=\prod_{i=1}^{m}\prod_{s=0}^{r-1}\left[ t_{i,s}^u\prod_{j=1}^m(t_{i,s}-\alpha_j)^{rn}\right]\prod_{(i_1,s_1)<(i_2,s_2)}(t_{i_2,s_2}-t_{i_1,s_1})\enspace,
$$
where the order $(i_1,s_{1})<(i_2,s_{2})$ means lexicographical order.
{{In the following of this section, the index $n$ will be conveniently omitted, to be easier to read.}}

Also set (when no confusion is deemed to occur, we omit the subscripts $\underline{\alpha}=(\alpha_1,\ldots,\alpha_m)$):
$$\Psi=\Psi_{\underline{\alpha}}:=\bigcirc_{i=1}^{m}\bigcirc_{s=0}^{r-1}{{\tilde{\psi}}}_{\alpha_i,i,s}\enspace.$$
Note that, by the definition of 
$\Theta$ (see \eqref{Theta}),
we have
\begin{align} \label{rn}
\Theta=\dfrac{\prod_{i=1}^m\alpha^r_i\prod_{s=0}^{r-1}a_{0,s}^m}{(n-1)!^{r^2m}}\Psi({\hat{P}}_{n,n})\enspace.
\end{align}
Let $u$ be a non-negative integer, we study the value
\begin{align}\label{Cn,u,m}
{{C_{n,u,m}=}}C_{u,m}:=\Psi({\hat{P}}_{u})\enspace.
\end{align} 
The following of subsection, we occupy the proof of the following property of $C_{u,m}$.
\begin{proposition} \label{decompose Cnum}
There exists a constant $c_{u,m}\in K$ with
$$
C_{u,m}=c_{u,m}\prod_{i=1}^m\alpha^{ru+r^2n+\binom{r}{2}}_i \prod_{1\le i_1<i_2\le m}(\alpha_{i_2}-\alpha_{i_1})^{(2n+1)r^2}\enspace.
$$
\end{proposition}
It is also easy to see that since all the variables $t_{i,s}$ have been specialized, $C_{u,m}\in K$ is a polynomial in the $\alpha_i$. The statement is then about a factorization of this polynomial. 
To prove of Proposition $\ref{decompose Cnum}$, we are going to perform the following steps~:
\begin{itemize}
\item[$({{a}})$] Show that $C_{u,m}$ is homogeneous of degree $m[r(u+1)+r^2n+{r\choose 2}]+{m\choose 2}(2n+1)r^2$.
\item[$({{b}})$] Show that $\prod_{i=1}^m\alpha^{r(u+1)+r^2n+\binom{r}{2}}_i$ divides $C_{u,m}$.
\item[$({{c}})$] Show that $\prod_{1\le i_1<i_2\le m}(\alpha_{i_2}-\alpha_{i_1})^{(2n+1)r^2}$ divides $C_{u,m}$.
\end{itemize}
We first prove $({{a}})$ and $({{b}})$.
\begin{lemma} \label{homo}
$C_{u,m}$ is homogeneous of degree $m[r(u+1)+r^2n+{r\choose 2}]+{m\choose 2}(2n+1)r^2$ and is divisible by
$\prod_{i=1}^m\alpha^{r(u+1)+r^2n+\binom{r}{2}}_i$.
\end{lemma}
\begin{proof} 
First the polynomial $\hat{P}_{u}(\boldsymbol{t})$ is a homogeneous polynomial with respect to the variables $\alpha_i,t_{i,s}$ of degree $m[ru+r^2n+{r\choose 2}]+{m\choose 2}(2n+1)r^2$.
By the definition of $\Psi$, it is easy to see that $C_{u,m}=\Psi(\hat{P}_{u}(\boldsymbol{t}))$ is a homogeneous polynomial with respect to the variables $\alpha_i$ of degree $m[r(u+1)+r^2n+{r\choose 2}]+{m\choose 2}(2n+1)r^2$.
Second we show the later assertion. By linear algebra ${{\tilde{\psi}}}_{\alpha,i,s}(B(t_{i,s}))=\alpha {{\tilde{\psi}}}_{1,i,s}(B(\alpha t_{i,s}))$ ({\it  i.e.}  the variable $t$ specializes in $1$, {\it confer} Lemma~\ref{prelimi} $(ii)$ below) for any integer $\ell$ and any polynomial $B(t_{i,s})\in K[t_{i,s}]$. 
So, by composition, the same holds for $\Psi$, and, putting $\boldsymbol{1}=(1,\ldots,1)$, {one gets}
$$C_{u,m}=\prod_{i=1}^m\alpha^r_i\cdot \bigcirc_{i=1}^m\left(\Psi_{\boldsymbol{1}}(\hat{P}_{u}(\alpha_i t_{i,s}))\right)\enspace.$$
We now compute
$$
\hat{P}_{u}(\alpha_i t_{i,s})=\prod_{i=1}^m\alpha_i^{r^2n+ru+{r\choose 2}}\cdot  Q_u({\boldsymbol{t}})\enspace,
$$
where 
\begin{align*}
Q_u(\boldsymbol{t})=Q_{n,u,m}(\boldsymbol{t})&=\left( \prod_{i=1}^m \prod_{s=0}^{r-1} \left[t_{i,s}^u\prod_{j\neq i}(\alpha_it_{i,s}-\alpha_j)^{rn}(t_{i,s}-1)^{rn}\right]\right)\\
& \cdot \displaystyle{\prod_{1\le i_1<i_2\le m}\prod_{0 \le s_1,s_2\le r-1}}(\alpha_{i_2}t_{i_2,s_2}-\alpha_{i_1}t_{i_1,s_1})\cdot \displaystyle{\prod_{i=1}^m\prod_{0\le s_1<s_2\le r-1}}(t_{i,s_2}-t_{i,s_1})\enspace,
\end{align*}
by linearity, we obtain
\begin{equation} \label{equality 1} 
C_{u,m}=\prod_{i=1}^m\alpha_i^{r(u+1)+r^2n+{r\choose 2}}\left(\Psi_{\boldsymbol{1}}\left(Q_u\right)\right)\enspace.
\end{equation}
This concludes the proof of the lemma. 
\end{proof}

{Now we consider $({{c}})$. Since the statement is trivial for $m=1$, we can assume $m\geq 2$. We need to show that $(\alpha_j-\alpha_i)^{(2n+1)r^2}$ divides $C_{u,m}$. Without loss of generality, after renumbering, we can assume that $j=2,i=1$.
To ease notations, we are going to take advantage of the fact that $m\geq2$, and set $X_s=t_{1,s}, Y_s=t_{2,s}$, $\alpha_1=\alpha$, $\alpha_2=\beta$ and ${{\tilde{\psi}}}_{\alpha_1,1,s}={{\tilde{\psi}}}_{\alpha,s},{{\tilde{\psi}}}_{\alpha_2,1,s}={{\tilde{\psi}}}_{\beta,s}$.
So our polynomial ${\hat{P}}_u$ rewrites as
\begin{align*}
{\hat{P}}_{u}(\underline{X},\underline{Y})&=\prod_{s=0}^{r-1}\left[(X_sY_s)^u[(X_s-\alpha)(X_s-\beta)(Y_s-\alpha)(Y_s-\beta)]^{rn}\right]\\
& \cdot \prod_{{{0}}\leq i<j\leq {{r-1}}}(X_j-X_i)\prod_{{{0}}\leq i<j\leq {{r-1}}}(Y_j-Y_i)\prod_{{{0}}\leq i,j\leq {{r-1}}}(Y_j-X_i)
\\ & \cdot c(t_{i,s})_{i\geq 3}\prod_{k\geq 3}\prod_{s=0}^{r-1}\prod_{{{0}}\leq i,j \leq {{r-1}}}(t_{k,s}-X_i)(t_{k,s}-Y_j)\enspace,
\end{align*}
where $c(t_{i,s}):=\prod_{i=3}^{m}\prod_{s=0}^{r-1}\left[ t_{i,s}^u\prod_{j=1}^m(t_{i,s}-\alpha_j)^{rn}\right]\prod_{(i_1,s_1)<(i_2,s_2),i_1,i_2\geq 3}(t_{i_2,s_2}-t_{i_1,s_1})$ (the precise value of $c$ does not actually matter as it is treated as a scalar by the operators ${{\tilde{\psi}}}_{\alpha,s},{{\tilde{\psi}}}_{\beta,s}$).

We set $\Psi_{\alpha}=\bigcirc_{s=0}^{r-1}{{\tilde{\psi}}}_{\alpha,s}$ and 
$\Psi_{\beta}=\bigcirc_{s=0}^{r-1}{{\tilde{\psi}}}_{\beta,s}$ respectively. 
One has 
$\Psi=\Psi_{\alpha}\circ\Psi_{\beta}\circ{{\underline{{{\tilde{\psi}}}}}}$
where ${{\underline{{{\tilde{\psi}}}}}}=\bigcirc_{i\geq 3}\bigcirc_{s=0}^{r-1}{{\tilde{\psi}}}_{\alpha_i,i,s}$. 

\begin{lemma}\label{prelimi}
\begin{itemize}
\item[$(i)$] The morphisms ${{\tilde{\psi}}}_{\alpha,s_1},{{\tilde{\psi}}}_{\beta,s_2}$ pairwise commute for $0\leq s_1,s_2\leq r-1$.
\item[$(ii)$] The operator $\frac{\partial}{\partial \alpha}$ commutes with any ${{\tilde{\psi}}}_{\beta,s}$ and hence with $\Psi_{\beta}$ and with ${{\underline{{{\tilde{\psi}}}}}}$.
\item[$(iii)$] We have 
$$\dfrac{\partial}{\partial\alpha}\Psi_{\alpha}=\Psi_{\alpha}\dfrac{\partial}{\partial \alpha}+\dfrac{1}{\alpha}\sum_{s=0}^{r-1}\Psi_{\alpha}\circ \theta_{X_s}\enspace.$$
\end{itemize}
\end{lemma}
\begin{proof} The assertion $(i)$ follows from the definition since both multiplication by a scalar, specialization of one variable or integration with respect to a given variable all pairwise commute, $(ii)$ follows from commutation of integrals with respect to a parameter with differentiation with respect to that parameter.
 
Finally, we prove $(iii)$.  If ${{h}}(X_0,\ldots,X_{r-1})\in K[\alpha,1/\alpha][X_0,\ldots, X_{r-1}]$, we write 
${{h}}(X_0,\ldots, X_{r-1})=\sum_{\underline{i}}a_{\underline{i}}(\alpha)\prod_{s=0}^{r-1}X_{s}^{i_s}$. 
By definition, we have
$$\Psi_{\alpha}({{h}})=\sum_{\underline{i}}a_{\underline{i}}(\alpha)\frac{\alpha^{\vert \underline{i}\vert}}{\prod_{s=0}^{r-1}(i_s+\gamma_{r-s})\cdots (i_s+\gamma_r)}\enspace,$$ where $\vert \underline{i}\vert=\sum_{s=0}^{r-1}i_s$. 
Then $$\frac{\partial}{\partial \alpha}(\Psi_{\alpha}({{h}}))=\sum_{\underline{i}}\frac{\partial}{\partial \alpha}(a_{\underline{i}}(\alpha))\frac{\alpha^{\vert \underline{i}\vert}}{\prod_{s=0}^{r-1}(i_s+\gamma_{r-s})\cdots (i_s+\gamma_r)}
+\sum_{\underline{i}}\dfrac{a_{\underline{i}}(\alpha)}{\alpha}\vert \underline{i}\vert\frac{\alpha^{\vert\underline{i}\vert}}{\prod_{s=0}^{r-1}(i_s+\gamma_{r-s})\cdots (i_s+\gamma_r)}\enspace.$$

First term in the sum is easily seen to be equal to $\Psi_{\alpha}\circ \frac{\partial}{\partial \alpha}({{h}})$.  
So the claim $(iii)$ reduces to the statement
$$\sum_{s=0}^{r-1}\Psi_{\alpha}\circ \theta_{X_s}(X^{i_0}_0\cdots X^{i_{r-1}}_{r-1})=
\vert \underline{i}\vert\frac{\alpha^{\vert\underline{i}\vert}}{\prod_{s=0}^{r-1}(i_s+\gamma_{r-s})\cdots (i_s+\gamma_r)}\enspace.$$

But left hand side is 
\begin{align*}
\sum_{s=0}^{r-1}\Psi_{\alpha}\circ\theta_{X_s}(X^{i_0}_0\cdots X^{i_{r-1}}_{r-1})=
\sum_{s=0}^{r-1}\frac{i_s\alpha^{\vert \underline i\vert}} {\prod_{s'=0}^{r-1}(i_{s'}+\gamma_{r-s'})\cdots (i_{s'}+\gamma_r)}
=\vert \underline{i}\vert\frac{\alpha^{\vert\underline{i}\vert}}{\prod_{s'=0}^{r-1}(i_{s'}+\gamma_{r-s'})\cdots (i_{s'}+\gamma_r)}\enspace.
\end{align*} This completes the proof of this lemma.
\end{proof}

\vspace{\baselineskip}
{{We introduce a specialization morphism for the variable $\alpha$. Set 
$$\Delta=\Delta_{\alpha}: \Q[\alpha,1/\alpha, \beta, \alpha_2,\ldots,\alpha_m] \longrightarrow \qu(\alpha); \ \ \Delta(P(\alpha,\beta,\alpha_2,\ldots,\alpha_m))=P(\alpha,\alpha,\alpha_2,\ldots,\alpha_m)\enspace.$$}}
Note that ${{\underline{{{\tilde{\psi}}}}}}$ and ${{\Psi_{\alpha}}}$ commute so, it is enough to prove that  
$$\frac{\partial^{\ell}}{\partial \alpha^{\ell}}\Delta_{\alpha} \left(\Psi_{\alpha}\circ\Psi_{\beta}({\hat{P}}_u)\right)=0  \ \ \text{for} \  0\le \ell \le (2n+1)r^2-1\enspace.$$
We postpone the end of the proof of $({{c}})$ and start with a few preliminaries.
We  set 
\begin{align} 
f_{n,u}(\alpha,\beta,\underline{X},\underline{Y})=f(\alpha,\beta,\underline{X},\underline{Y}) &=\displaystyle c(t_{i,s})\prod_{k\geq 3}\prod_{s{{=0}}}^{{r-1}}\prod_{1\leq i,j \leq r}(t_{k,s}-X_i)(t_{k,s}-Y_j) \label{f} \\
& \cdot \displaystyle \prod_{s={{0}}}^{{r-1}}(X_sY_s)^u[(X_s-\alpha)(X_s-\beta)(Y_s-\alpha)(Y_s-\beta)]^{rn}\enspace, \nonumber
\end{align}
and 
\begin{align} \label{g}
g(\underline{X},\underline{Y})=g(\alpha,\beta,\underline{X},\underline{Y})=\prod_{{{0}}\leq i,j\leq {{r-1}}}(X_i-Y_j)\prod_{{{0}}\leq i<j\leq {{r-1}}}[(X_j-X_i)(Y_j-Y_i)]\enspace.
\end{align}
So that ${\hat{P}}_{u}={\hat{P}}=fg$ (for the rest of the proof, the index $u$ will not play any role and may be conveniently left off to ease reading).

We now concentrate on a few elementary properties of the maps ${{\tilde{\psi}}}$ which we regroup here and will be useful for the rest~:

Now we prepare new notations. 
Let $\boldsymbol{\xi}_{s}:=(\xi_{s,k})_{k\ge 1}$ for $0\le s \le r-1$ be infinite sequences of elements of $K$. 
Put $\boldsymbol{\Xi}=(\boldsymbol{\xi}_{s})_{0\le s \le r-1}$.
For $\boldsymbol{\ell}:=(\ell_0,\ldots, \ell_{r-1})\in \Z^r$ {{with $\ell_i\ge0$}}, we put 
\begin{align*}
&{{\tilde{\psi}}}_{\alpha,s,\boldsymbol{\xi}_s,\ell_s}={{\tilde{\psi}}}_{\alpha,s}\bigcirc_{w=1}^{\ell_s}(\theta_{X_s}+\xi_{s,w}) \ \text{for} \ 0\le s \le r-1\enspace,\\
&\Psi_{\alpha,\boldsymbol{\Xi}_{\boldsymbol{\ell}}}=\bigcirc_{s=0}^{r-1}{{\tilde{\psi}}}_{\alpha,s,\boldsymbol{\xi}_s,\ell_s}\enspace,
\end{align*}
where $\bigcirc_{{{w}}=1}^{\ell_s}(\theta_{X_s}+\xi_{{s,w}})={\rm{id}}_{K[t]}$ if $\ell_s=0$.
We remark that, in the case of $\boldsymbol{\ell}=(0,\ldots,0)\in \Z^r$, we have $\Psi_{\alpha,\boldsymbol{\Xi}_{\boldsymbol{\ell}}}=\Psi_{\alpha}$ for any $\boldsymbol{\Xi}$.
\begin{lemma}\label{symetries}
Let $\boldsymbol{\xi}_{s}:=(\xi_{s,k})_{k\ge 1}$ and $\boldsymbol{\xi}'_{s}:=(\xi'_{s,k})_{k\ge 1}$ for $0\le s \le r-1$ be infinite sequences of elements of $K$
and $\boldsymbol{\ell}:=(\ell_0,\ldots,\ell_{r-1}),\boldsymbol{\ell}':=(\ell'_0,\ldots,\ell'_{r-1})\in \Z^r$ {{with $\ell_i,\ell'_j\ge0$}}.
Put $\boldsymbol{\Xi}:=(\boldsymbol{\xi}_s)_{s}$, $\boldsymbol{\Xi}':=(\boldsymbol{\xi}'_s)_s$. 
Assume there exist $\ell_i,\ell'_j$ with 
\begin{align} \label{equality theta}
\bigcirc_{w=0}^i (\theta_t+\gamma_{r-i+w})^{-1} \circ (\theta_t+\xi_{i,1})\circ \cdots \circ (\theta_t+\xi_{i,\ell_i})=
\bigcirc_{w'=0}^j(\theta_t+\gamma_{r-j+w'})^{-1} \circ (\theta_t+\xi'_{j,1})\circ \cdots \circ (\theta_t+\xi'_{j,\ell'_j})\enspace,
\end{align}
and the polynomial $P\in K[\underline{X},\underline{Y}]$ is antisymmetric $($any odd permutation of the variables $X_i,Y_j$ changes $P$ in its opposite$)$. 
Then we have
$$\Delta\circ\Psi_{\alpha,\boldsymbol{\Xi}_{\boldsymbol{\ell}}}\circ\Psi_{\beta,\boldsymbol{\Xi}'_{\boldsymbol{\ell}'}}(P)=0\enspace.$$
Similarly, if there exist $\ell_i,\ell_j$ for $0\le i<j \le r-1$ with 
\begin{align} \label{equality theta 2} 
\bigcirc_{w=0}^i(\theta_t+\gamma_{r-i+w})^{-1} \circ (\theta_t+\xi_{i,1})\circ \cdots \circ (\theta_t+\xi_{i,\ell_i})=
\bigcirc_{w'=0}^j (\theta_t+\gamma_{r-j+w'})^{-1}\circ (\theta_t+\xi_{j,1})\circ \cdots \circ (\theta_t+\xi_{j,\ell_j})\enspace,
\end{align}
we have $$\Psi_{\alpha,\boldsymbol{\Xi}_{\boldsymbol{\ell}}}(P)=0\enspace.$$
\end{lemma}
\begin{proof} Let $0\leq i,j\leq r-1$. 
Let $\tau$ be the transposition $\tau(X_i)=Y_j$, $\tau(Y_j)=X_i$ leaving all the other variables invariant.
Then $\tau$ acts on $K[\underline{X},\underline{Y}]$ by permutation of the variables. Then we have $\tau(P)=-P$ by antisymmetry. 
We compute
\begin{align*} 
\Delta\circ\Psi_{\alpha,\boldsymbol{\Xi}_{\boldsymbol{\ell}}}\circ\Psi_{\beta,\boldsymbol{\Xi}'_{\boldsymbol{\ell}'}}(P)&=\Delta\bigcirc_{s=0}^{r-1}{{\tilde{\psi}}}_{\alpha,s,\boldsymbol{\xi}_s,\ell_s}\bigcirc_{s=0}^{r-1}{{\tilde{\psi}}}_{\beta,s,\boldsymbol{\xi}^{'}_s,\ell'_s}(P)\\
&=\Delta\bigcirc_{s=0}^{r-1}{{\tilde{\psi}}}_{\alpha,s,\boldsymbol{\xi}_s,\ell_s}\bigcirc_{s=0}^{r-1}{{\tilde{\psi}}}_{\beta,s,\boldsymbol{\xi}^{'}_s,\ell'_s}(\tau P)=
-\Delta\circ\Psi_{\alpha,\boldsymbol{\Xi}_{\boldsymbol{\ell}}}\circ \Psi_{\beta,\boldsymbol{\Xi}'_{\boldsymbol{\ell}'}}(P)\enspace.
\end{align*}
Note that the second equality is obtained by the assumption $(\ref{equality theta})$. Thus we obtain the first assertion.
The second statement is a variation of the same argument.
\end{proof} 
\begin{remark}
Later, in Lemma $\ref{condition}$, we use the first assertion of Lemma $\ref{symetries}$ only to the case of $\boldsymbol{\ell}'=(0,\ldots,0)$. 
Namely, we apply Lemma $\ref{symetries}$ to the case of $\Psi_{\beta,\boldsymbol{\Xi}'_{\boldsymbol{\ell}'}}=\Psi_{\beta}$.
\end{remark}

\begin{lemma}\label{racine} 
Let $P\in K[\underline{X},\underline{Y}]$ be a polynomial such that $(X_s-\alpha)^T\mid P$ for some $T\geq 1$ and $0\leq \ell\leq T-1$ an integer. 
Let $\xi_1,\ldots, \xi_{\ell}\in K$ $($if $\ell=0$, we mean $\{\xi_1,\ldots,\xi_{\ell}\}=\emptyset)$.
Then we have $${{\tilde{\psi}}}_{\alpha,s}\bigcirc_{w'=0}^s (\theta_{X_s}+\gamma_{r-s+w'})\bigcirc_{w=1}^{\ell} (\theta_{X_s}+\xi_w)(P)=0\enspace.$$ 
\end{lemma} 
\begin{proof} Indeed, writing $P=(X_i-\alpha)^TQ$, with $Q\in K[\underline{X},\underline{Y}]$, and noting that  
\begin{align*}
{{\tilde{\psi}}}_{\alpha,s}\bigcirc_{w'=0}^s (\theta_{X_s}+\gamma_{r-s+w'})\bigcirc_{w=1}^{\ell} (\theta_{X_s}+\xi_w)(P)=
{\rm{Eval}}_{{{X_s\rightarrow \alpha}}}\bigcirc_{w=1}^{\ell} (\theta_{X_s}+\xi_w)(P)\enspace.
\end{align*} 
By the Leibniz formula and the hypothesis $\ell\leq T-1$, $\bigcirc_{w=1}^{\ell} (\theta_{X_s}+\xi_w)(P)$ belongs to the ideal $(X_s-\alpha)$ and so 
$\ev_{{{X_s\rightarrow \alpha}}}\bigcirc_{w=1}^{\ell} (\theta_{X_s}+\xi_w)(P)=0$.
\end{proof}
\begin{lemma}\label{racinesymetrique} 
Let  $P\in K[\underline{X},\underline{Y}]$ be a polynomial such that $((X_s-\alpha)^{T_1}(X_s-\beta)^{T_2})\mid P$ for some non-negative integers $T_1,T_2$ with either $T_1$ or $T_2$ is greater than $1$ and $0\leq \ell\leq T_1+T_2-1$ an integer. 
Let $\xi_1,\ldots, \xi_{\ell}\in K$ $($if $\ell=0$, we mean $\{\xi_1,\ldots,\xi_{\ell}\}=\emptyset)$.
Then, we have $$\Delta\circ{{\tilde{\psi}}}_{\alpha,s}\bigcirc_{w'=0}^s (\theta_{X_s}+\gamma_{r-s+w'})\bigcirc_{w=1}^{\ell} (\theta_{X_s}+\xi_w)(P)=0\enspace.$$
\end{lemma}
\begin{proof} 
This is a variation of the previous lemma, indeed, specialization at $\beta=\alpha$ doubles the multiplicity and commutation of specialization along $\beta$ commutes with 
${{\tilde{\psi}}}_{\alpha,s}\bigcirc_{w'=0}^s (\theta_{X_s}+\gamma_{r-s+w'})\bigcirc_{w=1}^{\ell} (\theta_{X_s}+\xi_w)$ (variation of Lemma~\ref{prelimi} $(i)$).
\end{proof}

Now, let us compute what comes out by iteration of property $(iii)$ of Lemma~\ref{prelimi}. 
We define infinite sequences of elements of $K$, $\boldsymbol{\xi}_s=(\xi_{s,k})_{k\ge 1}$, with
$$
\xi_{s,k}=\begin{cases}
\gamma_{r-s-1+k} & \ \text{if} \ 1\le k \le s+1\\
0 & \ \text{if} \ k>s+1\enspace,
\end{cases}
$$
and put $\boldsymbol{\Xi}=(\boldsymbol{\xi}_s)_{s=0,\ldots,r-1}$.
For a non-negative integer $\ell$, there exists a sequence $(b_{s,k,\ell})_{k=0,1,\ldots,\ell}\in K^{\ell+1}$ with $X^{\ell}=\sum_{k=0}^{\ell} b_{s,k,\ell}\prod_{w=1}^k(X+\xi_{s,w})$ where $\prod_{w=1}^k(X+\xi_{s,w})=1$ if $k=0$.

\

Let $\ell,k$ be {{non-negative integers}} with $\ell\ge k$. We define a set of differential operators 
$$\mathcal{X}_{\ell,k}=\{V=\partial_1\circ \cdots \circ \partial_{\ell}\mid \partial_i\in \{1/\alpha, \tfrac{\partial}{\partial \alpha}\}, \ \#\{1\le i \le \ell , \partial_i=1/\alpha\}=k\}\enspace.$$
One gets that 
\begin{align*}
\frac{\partial^{\ell}}{\partial \alpha^{\ell}}(\Psi_{\alpha}({\hat{P}}))&=\sum_{\substack{\boldsymbol{\ell}=(\ell_0,\ldots,\ell_{r-1})\in \Z^r \\ {{\ell_i\ge0}}, \ \vert \boldsymbol{\ell}\vert\leq \ell}}\sum_{V\in \mathcal{X}_{\ell,\vert \boldsymbol{\ell}\vert }}\Psi_{\alpha}\bigcirc_{s=0}^{r-1} \theta^{\ell_s}_{X_s}(V({\hat{P}}))
\\
&=\sum_{\substack{\boldsymbol{\ell}=(\ell_0,\ldots,\ell_{r-1})\in \Z^r \\ {{\ell_i\ge0}}, \ \vert \boldsymbol{\ell}\vert\leq \ell}}
\sum_{V\in \mathcal{X}_{\ell,\vert \boldsymbol{\ell}\vert }}\left(\sum_{k_0=0}^{\ell_0}\cdots\sum_{k_{r-1}=0}^{\ell_{r-1}}\prod_{s'=0}^{r-1}b_{s',k_{s'},\ell_{s'}}\Psi_{\alpha}\bigcirc_{s=1}^{r-1}\bigcirc_{u_s=1}^{k_s}(\theta_{X_s}+\xi_{s,u_s})(V({\hat{P}}))\right)\\
&=\sum_{\substack{\boldsymbol{\ell}=(\ell_0,\ldots,\ell_{r-1})\in \Z^r \\ {{\ell_i\ge0}}, \ \vert \boldsymbol{\ell}\vert\leq l}}\sum_{V\in \mathcal{X}_{\ell,\vert \boldsymbol{\ell}\vert }}
\left(\sum_{\substack{\boldsymbol{k}=(k_0,\ldots,k_{r-1})\le \boldsymbol{\ell} \\ k_i\ge 0}}\prod_{s'=0}^{r-1}b_{s',k_{s'},\ell_{s'}}\Psi_{\alpha,\boldsymbol{\Xi}_{\boldsymbol{k}}}(V({\hat{P}}))\right)\enspace,
\end{align*}
{{where $\boldsymbol{k}\le \boldsymbol{\ell}$ means $k_i\le l_i$ for each $0\le i \le r-1$.}}

By the Leibniz formula, for $V\in \mathcal{X}_{\ell,\vert \boldsymbol{\ell} \vert}$, $V({\hat{P}})$ is a linear combination (over $K[1/\alpha]$) of the derivatives 
$\frac{\partial^j}{\partial{\alpha^j}}({\hat{P}})$ for $0\leq j\leq \ell-\vert \boldsymbol{\ell}\vert$.
Since ${\hat{P}}=fg$ ({{recall the definition of $f$ and $g$ in \eqref{f} and \eqref{g} respectively}}), it is a linear combination of $g\frac{\partial^j}{\partial\alpha^j}(f)$, for $0\leq j\leq \ell-\vert\boldsymbol{\ell}\vert$.

\

We now perform the combinatorics argument~:
\begin{lemma} \label{condition}
We use the notations as above.
Let $\ell$ be a {{non-negative integer}} and $\boldsymbol{\ell}=(\ell_0,\ldots,\ell_{r-1})\in \Z^r$ {{with $\ell_i\ge0$}} such that $\vert \boldsymbol{\ell}\vert \leq \ell$.
Assume further either of these three to be true 
\begin{itemize}\label{condicombi}
\item[$(i)$]  There exist $0\le i \le r-1$ with $\ell_i<i+1$.

\item[$(ii)$] There exist $0\le i<j\le r-1$ with $\ell_i\ge i+1$, $\ell_j\ge j+1$ and $\ell_i-(i+1)=\ell_j-(j+1).$ 

\item[$(iii)$] There exists an index $0\leq s\leq r-1$ such that $0\leq \ell_s-(s+1)< 2rn-\ell+\vert\boldsymbol{\ell}\vert$.
\end{itemize}
Then, $\Delta \circ\Psi_{\alpha,\boldsymbol{\Xi}_{\boldsymbol{\ell}}}\circ \Psi_{\beta}(g\frac{\partial^jf}{\partial \alpha^j})=0$ for all $0\leq j\leq \ell-\vert\boldsymbol{\ell}\vert$.
\end{lemma}
\begin{proof} If the first condition is satisfied, we have
$$(\theta_t+\gamma_{r-i})^{-1}\circ \cdots \circ (\theta_t+\gamma_r)^{-1}\circ (\theta_t+\xi_{i,1})\circ \cdots \circ (\theta_t+\xi_{i,\ell_i})=
(\theta_t+\gamma_{r-i+\ell_i})^{-1}\circ \cdots \circ (\theta_t+\gamma_r)^{-1}\enspace.$$
Thus, by antisymmetry of $g$, the first assertion of Lemma~\ref{symetries} ensures vanishing.

If the second conditions are satisfied, we have 
\begin{align*}
\theta^{\ell_i-i-1}_t&=\bigcirc_{\ell=0}^{i}(\theta_t+\gamma_{r-i+\ell})^{-1}\circ(\theta_t+\xi_{i,1})\circ \cdots \circ (\theta_t+\xi_{i,\ell_i})\\
&=\bigcirc_{\ell=0}^{j}(\theta_t+\gamma_{r-j+\ell})^{-1}\circ(\theta_t+\xi_{j,1})\circ \cdots \circ (\theta_t+\xi_{j,\ell_j})=\theta^{\ell_j-j-1}_t\enspace.
\end{align*}
By antisymmetry of $g$, the second assertion of Lemma~\ref{symetries} ensures vanishing. 

If the third condition is satisfied, Lemma~\ref{racinesymetrique} ensures that 
\begin{align*}
\Delta \circ{{\tilde{\psi}}}_{\alpha,s}\bigcirc_{w=1}^{k_s}(\theta_{X_s}+\xi_{s,w})(V({\hat{P}}))=
\Delta \circ{{\tilde{\psi}}}_{\alpha,s}\bigcirc_{w=1}^{s+1}(\theta_{X_s}+\gamma_{r-s+w-1})\circ \theta_{X_s}^{k_s-s-1}(V({\hat{P}}))\enspace,
\end{align*} 
itself vanishes for all $V\in \mathcal{X}_{\ell,\vert \boldsymbol{\ell} \vert}$ since $\left[(X_s-\alpha)(X_s-\beta)\right]^{rn}\mid f$ (so $\frac{\partial^j}{\partial \alpha^j}({\hat{P}})$ vanishes at $\alpha=\beta$ at order at least 
$2rn-j\geq 2rn-\ell+\vert\boldsymbol{\ell}\vert>\ell_s-(s+1)$).
\end{proof}

\begin{lemma} \label{combi}
The smallest integer $\ell$ for which there exists $\boldsymbol{\ell}=(\ell_0,\ldots,\ell_{r-1})$ with $\vert\boldsymbol{\ell}\vert\leq l$ with none of the conditions of Lemma $\ref{condicombi}$ are satisfied is $(2n+1)r^2$. 
\end{lemma}
\begin{proof} Assume conditions $(i), (ii)$ and $(iii)$ are false, then the set $\{\ell_s-s-1\}$ is at least
$\{2rn-\ell+\vert\boldsymbol{\ell}\vert;\ldots;2rn-\ell+\vert\boldsymbol{\ell}\vert+r-1\}$ and $\sum_{s=0}^{r-1}(\ell_s-s-1)\geq  r(2rn-\ell+\vert\boldsymbol{\ell}\vert)+r(r-1)/2$, that is 
$\vert \boldsymbol{\ell}\vert+r(\ell-\vert\boldsymbol{\ell}\vert)\geq2 r^2n+r^2$. 
Since $\ell-\vert\boldsymbol{\ell}\vert\geq 0$, the lemma follows.
\end{proof}

{\noindent \bf End of the proof of Proposition~\ref{decompose Cnum} $({{c}})$~:}

Lemma~\ref{combi} ensures that 
$$\frac{\partial^{\ell}}{\partial \alpha^{\ell}}\Delta_{\alpha}(C_{u,m})=0 \hspace{15pt} \text{for all} \ 0\leq \ell\leq (2n+1)r^2-1\enspace.$$ 
This completes the proof of Proposition $\ref{decompose Cnum}$. \qed

\subsection{Last step}
We shall reduce by induction the non-vanishing of $c_{u,m}$ to the non-vanishing of $c_{u,0}$ (which is obviously equal to 1). First, we prove,
\begin{lemma}\label{reduc} 
Set $\displaystyle {{\mathfrak{A}}}(\boldsymbol{t})=\kern-3pt{\prod_{s={{0}}}^{{{r-1}}}}\kern-3pt \left[t^{u}_{m,s}\kern-2pt \cdot\kern-2pt (t_{m,s}-1)^{rn}\right] \cdot \kern-10pt\prod_{1\le s<s'\le r}\kern-10pt(t_{m,s}-t_{m,s'})$ and ${{\mathcal{L}}}_m=\bigcirc_{1\leq s\leq r}{{\tilde{\psi}}}_{1,m,s} $. Then, 
$$c_{u,m}=(-1)^{r^2n(m-1)}c_{u+r(n+1),m-1}\cdot {{\mathcal{L}}}_m\left({{\mathfrak{A}}}(\boldsymbol{t}))\right)\enspace.$$
\end{lemma}
\begin{proof}
Set ${{\hat{\mathcal{L}}}}=\bigcirc_{i=1}^{m-1}\bigcirc_{1\leq i\leq m-1}\bigcirc_{1\leq s\leq r}{{\tilde{\psi}}}_{1,i,s}$ so that $\Psi_{\bold{1}}={{\hat{\mathcal{L}}}}\circ{{\mathcal{L}}}_m$ and
recall that by $(\ref{equality 1})$, 
$$ \label{deco Cnum2} 
D_{u,m}:=\dfrac{C_{u,m}}{\prod_{i=1}^m\alpha^{r(u+1)+r^2n+\binom{r}{2}}_i}=c_{u,m}\prod_{1\le i<j\le m}(\alpha_{j}-\alpha_{i})^{(2n+1)r^2}=\Psi_{\boldsymbol{1}}\left(Q_{u,m}\right)\enspace.
$$
We are going to evaluate $D_{u,m}$ at $\alpha_m=0$  and thus separate the variables in $Q_{u,m}$ first. By definition, one has
\begin{equation*}
Q_{u,m}(\boldsymbol{t})=Q_{u,m-1}(\boldsymbol{t})\cdot {{\mathfrak{A}}}(\boldsymbol{t}){{\mathfrak{B}}}(\boldsymbol{t})\enspace,
\end{equation*} 
where 
$${{\mathfrak{B}}}(\boldsymbol{t})=\prod_{s={{0}}}^{{{r-1}}}\prod_{j\neq m}(\alpha_mt_{m,s}-\alpha_j)^{rn}\cdot
\prod_{i=1}^{m-1} \prod_{s={{0}}}^{{r-1}} (\alpha_it_{i,s}-\alpha_m)^{rn}
\cdot  {\prod_{1\le i< m}\prod_{{{0}}\le s,s'\le {{r-1}}}}\kern-8pt(\alpha_{m}t_{m,s'}-\alpha_{i}t_{i,s})\enspace.
$$
Note that $Q_{u,m-1}, {{\mathfrak{A}}}$ do not depend on $\alpha_m$, and $\Psi_{\boldsymbol{1}}$ treats $\alpha_m$ as a scalar.
Hence,
\begin{align} \label{compare 1}
\left.{D_{u,m}}\right|_{\alpha_m=0} &= \displaystyle c_{u,m}\prod_{i=1}^{m-1}(-\alpha_i)^{(2n+1)r^2}\prod_{1\le i<j\le m-1}(\alpha_{j}-\alpha_{i})^{(2n+1)r^2}\enspace\\
&=\Psi_{\boldsymbol{1}}\left(Q_{u,m-1}(\boldsymbol{t}){{\mathfrak{A}}}(\boldsymbol{t})\left.{{\mathfrak{B}}}(\boldsymbol{t})\right|_{\alpha_m=0}\right)\enspace. \nonumber
\end{align}
But 
$$\left.{{\mathfrak{B}}}(\boldsymbol{t})\right|_{\alpha_m=0}\kern-1pt=\kern-3pt\prod_{j=1}^{m-1}\kern-2pt(-\alpha_j)^{r^2n}\kern-3pt\prod_{i=1}^{m-1}\kern-2pt\prod_{s={{0}}}^{{r-1}}(\alpha_it_{i,s})^{rn\kern-4pt}\prod_{i=1}^{m-1}\kern-3pt\prod_{s={{0}}}^{{{r-1}}}(\kern-1pt-\alpha_it_{i,s})^r\kern-2pt=\kern-2pt(-1)^{r^2(m-1)(n+1)}\kern-3pt\prod_{i=1}^{m-1}\alpha_i^{(2n+1)r^2}\kern-2pt\prod_{i=1}^{m-1}\kern-3pt\prod_{s={{0}}}^{{r-1}}t_{i,s}^{r(n+1)}\enspace\kern-3pt.$$
We now note  that $\theta_{m}$ treats the variables $t_{i,s}, 1\leq i\leq m-1$ as scalars and ${{\hat{\mathcal{L}}}}$ treats variables $t_{m,s}$ as scalars and remark $$Q_{u,m-1}(\boldsymbol{t})\!\left.{{\mathfrak{B}}}(\boldsymbol{t})\right|_{\alpha_m=0}=(-1)^{r^2(m-1)(n+1)}\prod_{i=1}^{m-1}\alpha_i^{(2n+1)r^2}Q_{u+r(n+1),m-1}{{(\boldsymbol{t})}}\enspace.$$
Thus
$$\Psi_{\boldsymbol{1}}\left(Q_{u,m-1}(\boldsymbol{t}){{\mathfrak{A}}}(\boldsymbol{t})\left.{{\mathfrak{B}}}(\boldsymbol{t})\right|_{\alpha_m=0}\right)
=(-1)^{r^2(m-1)(n+1)}\prod_{i=1}^{m-1}\alpha_i^{(2n+1)r^2}{{\hat{\mathcal{L}}}}(Q_{u+r(n+1),m-1}(\boldsymbol{t})){{\mathcal{L}}}_m({{\mathfrak{A}}}(\boldsymbol{t}))\enspace.$$
Using the relation~(\ref{compare 1}), taking into account $D_{u+r(n+1),m-1}=\Psi_{\boldsymbol{1}}(Q_{u+r(n+1),m-1}(\boldsymbol{t}))$ and simplifying,
$$c_{u,m}=
(-1)^{r^2n(m-1)}c_{u+r(n+1),m-1}\cdot {{\mathcal{L}}}_m({{\mathfrak{A}}}(\boldsymbol{t}))\enspace.$$
This completes the proof of Lemma $\ref{reduc}$.
\end{proof}
By Lemma $\ref{reduc}$, to prove the non-vanishing of the value $c_{u,m}$, it is enough to show ${{\mathcal{L}}}_m({{\mathfrak{A}}}(\boldsymbol{t}))\neq 0$.
Denote the cardinality of the set $\{\zeta_1,\ldots,\zeta_r\}$ by $d$. If we need, by changing the order, we may assume $\{\zeta_1,\ldots,\zeta_r\}=\{\zeta_1,\ldots,\zeta_d\}$ and 
$$(\zeta_1,\ldots,\zeta_r)=(\overbrace{\zeta_1,\ldots,\zeta_1}^{r_1},\ldots ,\overbrace{\zeta_{d},\ldots,\zeta_{d}}^{r_d})\enspace,$$
where $r_j$ is the multiplicity of $\zeta_j$ for $1\le j \le d$. For an integer $s$, we define the $K$-homomorphism $\varphi_{\zeta_j,s}$ by
$$\varphi_{\zeta_j,s}:K[t]\longrightarrow K; \ t^k\mapsto \dfrac{1}{(k+\zeta_j)^s}\enspace.$$
\begin{lemma}
There exists $E\in K\setminus\{0\}$ with 
\begin{align} \label{det last}
{{\mathcal{L}}}_m({{\mathfrak{A}}}(\boldsymbol{t}))=E\cdot {\rm{det}}\left(\varphi_{\zeta_j,s_j}(t^{u+\ell}(t-1)^{rn})\right)_{\substack{0\le \ell \le r-1 \\ 1\le j \le d, 1\le s_j \le r_j}}\enspace.
\end{align}
Especially, the value ${{\mathcal{L}}}_m({{\mathfrak{A}}}(\boldsymbol{t}))$ is not zero.
\end{lemma}
\begin{proof}
Define 
\begin{align} \label{psi s}
\psi_s: K[t]\longrightarrow K; \ t^k\mapsto \dfrac{1}{(k+\gamma_{r-s})\cdots(k+\gamma_r)}=\dfrac{1}{(k+\zeta_1)\cdots(k+\zeta_{s+1})}\enspace,
\end{align}
for $0\le s \le r-1$. Then we have 
\begin{align} \label{psi s}
{{\mathcal{L}}}_m({{\mathfrak{A}}}(\boldsymbol{t}))={\rm{det}}(\psi_s(t^{u+\ell}(t-1)^{rn}))_{\substack{0\le s \le r-1 \\ 0\le \ell \le r-1}}\enspace.
\end{align}
For $0\le s \le r-1$, there exist $1\le w \le d$ and $1 \le s_w \le r_w$ with 
\begin{align} \label{u,s_u}
s+1=r_1+\cdots+r_{w-1}+s_w\enspace.
\end{align}
Put 
$${{p}}_{j,k}=
\begin{cases}
\left. \dfrac{1}{(r_j-k)!}\dfrac{d^{r_j-k}}{dX^{r_j-k}}\dfrac{1}{\prod_{\substack{1\le j'\le w-1\\ j'\neq j}}(X+\zeta_{j'})^{r_{j'}}(X+\zeta_w)^{s_w}}\right|_{X=-\zeta_j} &  \ \text{if} \ \ 1\le j \le w-1, 1\le k \le r_j\enspace,\\
\left. \dfrac{1}{(s_w-k)!}\dfrac{d^{s_w-k}}{dX^{s_w-k}}\dfrac{1}{\prod_{j=1}^{w-1}(X+\zeta_j)^{r_j}}\right|_{X=-\zeta_{w}} & \ \text{if} \ \ j=w, 1\le k \le s_w\enspace.
\end{cases} 
$$
Then we have ${{p}}_{w,s_w}=\dfrac{1}{\prod_{j=1}^{w-1}(\zeta_j-\zeta_w)^{r_j}}\neq 0$ and 
\begin{align} \label{psi s 2}
\psi_s=\sum_{j=1}^{w-1}\sum_{k=1}^{r_j}{{p}}_{j,k}\varphi_{\zeta_j,k}+\sum_{k=1}^{s_w}{{p}}_{w,k}\varphi_{\zeta_w,k}\enspace.
\end{align}
Put $E=\prod_{0\le s \le r-1}{{p}}_{w,s_w}\neq 0$ where $(w,s_w)$ is the pair of integers defined as in $(\ref{u,s_u})$ for $s+1$. Then by equalities $(\ref{psi s})$, $(\ref{psi s 2})$ and the linearity of the determinant, we obtain $(\ref{det last})$.
{{The non-vanishing of the determinant $${\rm{det}}\left(\varphi_{\zeta_j,s_j}(t^{u+\ell}(t-1)^{rn})\right)_{\substack{0\le \ell \le r-1 \\ 1\le j \le d, 1\le s_j \le r_j}}$$
has been obtained in \cite[Proposition $4.12$]{DHK4}.}}
\end{proof}

\section{Estimates}
In this subsection, we use the following notations. 
Let $K$ an algebraic number field and $v$ be a place of $K$. Denote by $K_v$ the completion of $K$ at $v$, $\vert\cdot\vert_v$ the absolute value corresponding to $v$.
Let $\eta_1,\ldots,\eta_r,\zeta_1,\ldots,\zeta_r$ be strictly positive rational numbers with $\eta_i-\zeta_j\notin \N$ for $1 \le i,j \le r$.
Put $A(X)=(X+\eta_1)\cdots (X+\eta_r)$ and $B(X)=(X+\zeta_1)\cdots (X+\zeta_r)$.
We shall choose a sequence $\boldsymbol{c}:=(c_k)_{k\ge0}$ satisfying $c_k\in K\setminus\{0\}$ and \eqref{recurrence 1} for the given polynomials $A(X),B(X)$.
Let $\boldsymbol{\alpha}:=(\alpha_1,\ldots,\alpha_m)\in (K\setminus\{0\})^m$ whose coordinates are pairwise distinct and $n$ be a non-negative integer.
We choose $\gamma_1=\zeta_r,\ldots,\gamma_{r-1}=\zeta_2$.
For non-negative integer $\ell$ with $0\le \ell \le rm$, recall the polynomials $P_{\ell}(z), P_{\ell,i,s}(z)$ defined as in $(\ref{Pl})$ and $(\ref{Plis})$ for the given data.

Throughout the section, the small $o$-symbol  $o(1)$ and $o(n)$ refer when $n$ tends to infinity. 
Put $\varepsilon_v=1$ if $ v\mid \infty$ and $0$ otherwise.

Let $I$ be a non-empty finite set of indices, ${{R}}=K[\alpha_i]_{i\in I}[z,t]$ be a polynomial ring in indeterminate $\alpha_i,z,t$. For a non-negative integer $n$ and $\zeta\in \Q\setminus\{0\}$, we define $$S_{n,\zeta}: K[t]\longrightarrow K[t]; \ t^k\mapsto \dfrac{(k+\zeta+1)_n}{n!}t^k\enspace.$$ 

We set $\|P\|_v=\max\{\vert a\vert_v\}$ where $a$ runs in the coefficients of $P$. Thus ${{R}}$ is endowed with a structure of normed vector space. If $\phi$ is an endomorphism of ${{R}}$, we denote by $\| \phi\|_{v}$ the endomorphism norm defined in a standard way $\| \phi\|_v=\inf\{M\in\ru, \forall\kern3pt x\in {{R}}, \kern3pt\|\phi(x)\|_v\leq M\| x\|_v\}=\sup\left\{\frac{\|\phi(x)\|_v}{\| x\|_v}, 0\neq x\in {{R}}\right\}$. This norm is well defined provided $\phi$ is continuous. Unfortunately, we will have to deal  also with non-continuous morphisms. In such a situation, we restrict the source space to some appropriate sub-vector space $E$ of ${{R}}$ and talk of $\|\phi\|_v$ with $\phi$ seen  as $\left.\phi\right|_{E}: E\longrightarrow {{R}}$ on which $\phi$ is continuous. In case of perceived ambiguity, it will be denoted by $\| \phi\|_{E,v}$.
The degree of an element of ${{R}}$ is as usual the total degree. 

{{For a rational number $x$, we denote by $\lceil x \rceil$ the greatest integer less than or equal to $x$.}} 
\begin{lemma} \label{valuation} $(${\it confer} {\rm{\cite[Lemma $4.1$ $(ii)$]{KP}}}$)$
Let $a,b\in \Q$ {{which are not negative integers}}. 
For a {{non-negative integer}} $n$, put $$D_n={\rm{den}}\left(\dfrac{(a)_0}{(b)_0},\ldots,\dfrac{(a)_n}{(b)_n}\right)\enspace.$$
Then we have
$$\limsup_{n\to \infty}\dfrac{1}{n}\log\,D_n\le \log\,\mu(a)+\dfrac{{\rm{den}}(b)}{\varphi({\rm{den}}(b))}\enspace,$$
where $\varphi$ is the Euler's totient function.
\end{lemma}
\begin{proof} 
Put $$D_{1,n}={\rm{den}}\left(\dfrac{(a)_0}{0!},\ldots,\dfrac{(a)_n}{n!}\right)\enspace, \ \ D_{2,n}={\rm{den}}\left(\dfrac{0!}{(b)_0},\ldots,\dfrac{n!}{(b)_n}\right)\enspace.$$
Then we have $D_n\le D_{1,n}\cdot D_{2,n}$. 
The assertion is deduced from 
\begin{align}
&\limsup_{n\to \infty}\dfrac{1}{n}\log\,D_{1,n}\le \log\,\mu(a)\enspace, \nonumber \\ 
&\limsup_{n\to \infty}\dfrac{1}{n}\log\,D_{2,n}\le \dfrac{{\rm{den}}(b)}{\varphi({\rm{den}}(b))}\enspace. \label{D2n}
\end{align}
First inequality is proved in \cite[Lemma $2.2$]{B}.
Second inequality is shown in {\cite[Lemma $4.1$]{KP}}, however, we explain here this proof  in an abbreviated form, to let our article be self-contained.
This proof is originally indicated by Siegel \cite[p.81]{Siegel}. Put $d={\rm{den}}(b)$, $c=d\cdot b$.
Set $N_k:=c(c+d)\cdots(c+(k-1)d)$ for a non-negative integer $k$.
Let $p$ be a prime number with $p|N_k$. Then the following properties hold.

    \medskip

    $({\rm{a}})$  The integers $p,d$ are coprime and, for any integers $i, \ell$ with $\ell>0$, there exists exactly one integer $\nu$ with $0\le \nu \le p^{\ell}-1$ with $p^{\ell}|c+(i+\nu)d$.

    \medskip

    $({\rm{b}})$ Let $\ell$ be a strictly positive integer with $|c|+(k-1)d<p^{\ell}$. Then $N_k$ is not divisible by $p^{\ell}$.

    \medskip

    $({\rm{c}})$ Set $C_{p,k}:=\lfloor \log(|c|+(k-1)d)/\log(p)\rfloor $. Then we have
    \[
        v_p(k!)=\sum_{\ell=1}^{C_{p,k}}\left \lfloor \dfrac{k}{p^{\ell}}  \right\rfloor\le v_p(N_k)\le \sum_{\ell=1}^{C_{p,k}}\left(1+\left \lfloor \dfrac{k}{p^{\ell}} \right\rfloor  \right)=v_p(k!)+C_{p,k},
    \]
    where $v_p$ denotes the $p$-adic valuation. These relations imply
    \[
        \log\,\left|\dfrac{k!}{(\beta)_k}\right|_p\le
        \begin{cases}
         C_{p,k}\log(p) & \ \ \text{if} \ \ p \mid N_k\\
         0 & \ \ \text{otherwise}\enspace,
        \end{cases}
    \]
   and thus 
    \begin{align*}
        \log\,D_{2,n}=\sum_{p:\text{prime}}\max_{0\le k \le n} \log\,\left|\dfrac{k!}{(\beta)_k}\right|_p&\le \log(|c|+(n-1)d)\pi_{|c|,d}(|c|+(n-1)d)\enspace,
    \end{align*}
    where $\pi_{|c|,d}(x):=\#\{p:\text{prime}~;~ p\equiv |c| \ \text{mod} \ d, \ p<x\}$ for $x>0$. 
Finally, Dirichlet's prime number theorem for arithmetic progressions conclude \eqref{D2n}.  
\end{proof}

\begin{lemma} \label{fundamental}
\label{norme} We have the following norm estimates $($we do hope the similarity of notations is not cause of confusion$) :$
\begin{itemize}
\item[$(i)$]  Let $E_N$ be the subspace of ${{R}}$ consisting of polynomials of degree at most $N$ in $t$. Then for all $n\geq 1$ and {{strictly positive rational number}} $\zeta$, the morphism $S_{n,\zeta}$ satisfies 
$$\| S_{n,\zeta}\|_{E_N,v}\leq \begin{cases}
1  & \ \text{if} \ v\nmid \infty \ \text{and} \ |\zeta|_v\le 1\\
\left\vert \dfrac{(N+\zeta+1)_n}{n!} \right\vert_v   &  \ \text{otherwise}\enspace.
\end{cases}$$ 
It acts diagonally on ${{R}}$ in the sense that each element of the canonical basis consisting of all monomials is an eigenvector for $S_{n,\zeta}$. This map conserves degrees.
\item[$(ii)$]  Let $\zeta$ {{be a strictly positive rational number}}. Then the morphism $\theta_t+\zeta$ satisfies 
$$\| \theta_t+\zeta \|_{E_N,v}\leq 
\begin{cases} 1  & \ \text{if} \ v \ \text{is non-archimedian and} \ |\zeta|_v\le 1\\
\left\vert N+\zeta \right\vert_v & \ \text{otherwise}\enspace.
\end{cases}
$$ 
It acts diagonally on ${{R}}$ in the sense that each element of the canonical basis consisting of all monomials is an eigenvector for $\theta_t+\zeta$. This map conserves degrees.
\item[$(iii)$] Let $\boldsymbol{c}=(c_k)_{k\ge0}$ satisfying $c_k\in \Q\setminus\{0\}$ and $(\ref{recurrence 1})$ for $A(X)=(X+\eta_1)\cdots(X+\eta_r),B(X)=(X+\zeta_1)\cdots(X+\zeta_r)$.
For a {{non-negative integer}} $N$, we put $$D_{\boldsymbol{c},N}={\rm{den}}\left(\dfrac{(1+\zeta_1)_k\cdots (1+\zeta_r)_k}{(\eta_1)_k\cdots(\eta_r)_k}\right)_{0\le k \le N}\enspace.$$ 
The morphism ${\mathcal{T}}_{\bold{c}}$ which is defined in $(\ref{Phi})$ satisfies $$\| {\mathcal{T}}_{\bold{c}} \|_{E_N,v}\leq 
\begin{cases}
e^{o(N)} & \  \text{if} \ v\mid\infty \\
|c^{-1}_0D_{\boldsymbol{c},N}|^{-1}_v & \ \text{otherwise}\enspace,
\end{cases} 
$$ 
for $N\to \infty$. It acts diagonally on ${{R}}$ in the sense that each element of the canonical basis consisting of all monomials is an eigenvector for ${\mathcal{T}}_{\bold{c}}$. This map conserves degrees.
\end{itemize}
\end{lemma}
\begin{proof} 
$(i)$ and $(ii)$ follow from the very definition of $S_{n,\zeta}$ and $\theta_t+\zeta$.
We proof $(iii)$. Since we have $${\mathcal{T}}_{\bold{c}}(t^k)=\dfrac{1}{c_0}\dfrac{B(1)\cdots B(k)}{A(0)\cdots A(k-1)}=\dfrac{1}{c_0}\dfrac{(1+\zeta_1)_k\cdots (1+\zeta_r)_k}{(\eta_1)_k\cdots(\eta_r)_k}\enspace,$$ we get 
\begin{align} \label{upper Phi}
\| {\mathcal{T}}_{\bold{c}} \|_{E_N,v}\le \max_{0\le k \le N} \left(\left|\dfrac{1}{c_0}\dfrac{(1+\zeta_1)_k\cdots (1+\zeta_r)_k}{(\eta_1)_k \cdots(\eta_r)_k}\right|_v\right)\enspace. 
\end{align}
Let $v$ be an archimedian valuation. Since we have $$\dfrac{(1+\zeta_j)_k}{(\eta_j)_k}\le \dfrac{k}{\eta_j}\binom{k+\lceil \zeta_j\rceil}{k} \ \ (k\ge0)\enspace,$$ we obtain
$$\| {\mathcal{T}}_{\bold{c}} \|_{E_N,v}\le \left|\dfrac{N^r}{c_0\cdot \eta_1\cdots \eta_r}\binom{N+\lceil \zeta_1 \rceil}{N} \cdots \binom{N+ \lceil \zeta_r\rceil}{N} \right|_v=e^{o(N)} \ (N\to \infty)\enspace.$$
For a non-archimedian place $v$, by $(\ref{upper Phi})$ and the definition of $D_{\boldsymbol{c},N}$, we get the desire estimate.
\end{proof}

From the preceding lemma, we deduce~: 
\begin{lemma} \label{majonorme} 
For a {{strictly positive integer}} $N$, we put 
$$D_{\boldsymbol{c},N}={\rm{den}}\left(\dfrac{(1+\zeta_1)_k\cdots (1+\zeta_r)_k}{(\eta_1)_k\cdots(\eta_r)_k}\right)_{0\le k \le N}\enspace, \ \ D'_{\boldsymbol{c},N}={\rm{den}}\left(\dfrac{(\eta_1)_k\cdots(\eta_r)_k}{(1+\zeta_1)_k\cdots (1+\zeta_r)_k}\right)_{0\le k \le N}\enspace.$$
We denote by $w$ the place of $\Q$  such that  $v\vert w$. One has~$:$
\begin{itemize}
\item[$(i)$] The polynomial $P_{\ell}(z)=P_{n,\ell}(\boldsymbol{\alpha}\vert z)$ satisfies 
$$\| P_{\ell}(z)\|_v\leq 
\begin{cases}
\exp\left(\dfrac{n[K:_v:\Q_w]}{[K:\Q]}\left[\const+o(1)\rule{0mm}{4mm}\right]\right) & \text{if} \ v\mid\infty \\\rule{0mm}{7mm}
e^{o(1)}|D_{\boldsymbol{c},rmn}|^{-1}_v  \cdot \prod_{j=1}^r \left\vert\mu_{n-1}(\zeta_j)\right\vert_v^{{-{{1}}}} & \text{if} \ v\mid p \enspace,
\end{cases} 
$$ where $o(1)\longrightarrow 0$ for $n\rightarrow\infty$. 
Recall that $P_{\ell}(z)$ is of degree $rn$ in each variable $\alpha_i$, of degree $rmn+\ell$ in $z$ and constant in $t$.
\item[$(ii)$] 
The polynomial $P_{\ell,i,s}(z)=P_{\ell,i,s}(\boldsymbol{\alpha}|z)$ satisfies
$$
\| P_{\ell,i,s}(z)\|_v\leq   \begin{cases}
\exp\left({{\dfrac{{n}[K_v:\Q_w]}{[K:\Q]}}}\left[\const+o(1)\rule{0mm}{4mm}\right]\right) & \text{if} \ v\mid\infty \\
e^{o(1)}|D_{\boldsymbol{c},rmn}\cdot D'_{\boldsymbol{c},rmn}|^{-1}_v\cdot \prod_{j=1}^r\left\vert \mu_{n-1}(\zeta_j)\right\vert_v^{-{{1}}} & \text{if} \ v\mid p \enspace.
\end{cases}
$$
Also, $P_{\ell,i,s}(z)$ is of degree $\leq rmn+\ell$ in $z$, of degree $rn$ in each of the variables $\alpha_j$ except for the index $i$ where it is of degree $rn+1$ 
$($recall that $\psi_{i,s}$ involves multiplication by $[\alpha_i]$$)$.
\item[$(iii)$] For any integer $k\geq 0$, the polynomial $\psi_{i,s}\circ[t^{k+n}](P_{\ell}(t))$ satisfies 
{\small{
$$
\left\vert \psi_{i,s}\circ[t^{k+n}](P_{\ell}(t) {{)}}\right\vert_v\leq
\begin{cases}
\exp\left(\dfrac{n[K_v:\Q_w]}{[K:\Q]}\left[\const+o(1)\rule{0mm}{4mm}\right] \right)& \text{if} \ v\mid\infty \\
e^{o(1)}|D_{\boldsymbol{c},rmn}\cdot D'_{\boldsymbol{c},rmn}|^{-1}_v \cdot  \prod_{j=1}^r\left\vert\mu_{n-1}(\zeta_j)\right\vert_v^{-{{1}}}\rule{0mm}{7mm} & \text{otherwise}\enspace.
\end{cases}
$$}}
By definition, it is a homogeneous polynomial in  just the variables $\boldsymbol{\alpha}$ of degree $\leq rmn+\ell+k+n+1$.
\end{itemize}
\end{lemma}

\begin{proof} 
Let $I$ be of cardinality $m$,  $E_N$ be the sub-vector space of $K[y_1,\ldots,y_m, z,t]$ consisting of polynomials of degree at most $N$ in the variables $y_i$ and $\Gamma: E_N\longrightarrow {{R}}$ the morphism defined by $\Gamma(Q(y_1,\ldots,y_m,z,t))=Q(t-\alpha_1,\ldots,t-\alpha_m,z,t)$. 
Set $B_{n,l}(\boldsymbol{y},t)=t^{\ell}\prod_{i=1}^my_i^{rn}$,
since $B_{n,l}$ is a monomial, its norm $\| B_{n,l}\|_v=1$. 
By definition, one has $$P_{\ell}( z)=\ev_{t\rightarrow z}\circ {\mathcal{T}}_{\bold{c}}\bigcirc_{j=1}^r S_{n-1,\zeta_j}\circ\Gamma(B_{n,l})\enspace,$$ and thus, by sub-multiplicatively of the endomorphism norm,
\begin{align*}
\| P_{\ell}(z)\|_v
&\leq 2^{\varepsilon_v rmn [K_v:\Q_w]/[K:\Q]}{{e^{\varepsilon_v o(n)}}} |D_{\boldsymbol{c},N}|^{-1}_v\cdot
\begin{cases} 
\prod_{j=1}^r \left\vert \dfrac{(rmn+\ell+1+\zeta_j)_{n-1}}{(n-1)!}\right\vert_v & \ \ \text{if} \ v\mid\infty\\
& \\
\prod_{j=1}^r \left\vert \dfrac{(rmn+\ell+1+\zeta_j)_{n-1}}{(n-1)!}\right\vert^{{{\delta_v(\zeta_j)}}}_v & \ \ \text{otherwise}\enspace,
\end{cases}
\end{align*}
where 
$\delta_v(\zeta_j)=1$ if $|\zeta_j|_v>1$ and $0$ otherwise
(one can choose $N=rn$ while using \cite[Lemma $5.2$ $(iv)$]{DHK2} and $N=r(n+1)m+\ell$ for Lemma $\ref{fundamental}$ $(i)$ and $(iii)$ using $\ell\leq rm$, and note that the original polynomial is a constant in $z$ so the evaluation map is an isometry). 

\vspace{0.5\baselineskip} 

In the ultrametric case, we have the claimed result
$$\left|\dfrac{(rmn+rm+1+\zeta_j)_{n-1}}{(n-1)!}\right|_v\le |\mu_{n-1}(\zeta_j)|^{-1}_v\enspace,$$
where $\mu_{n}(\zeta_j):=\prod_{\substack{q:\rm{prime} \\ q\mid {\rm{den}}(\zeta_j)}}q^{n+\lceil n/(q-1) \rceil}$ \,
$(${\it confer} \rm{\cite[Lemma $2.2$]{B}}$)$) and Lemma $\ref{fundamental}$ $(iii)$ with $N=rmn+rm$.

Left to prove is the archimedian case, we put $Y=\max_j\{\lceil \zeta_j\rceil\}$. Then we have~:
$$\dfrac{(rmn+rm+1+\zeta_j)_{n-1}}{(n-1)!}\le \binom{(rm+1)n+rm+Y}{n-1}\enspace,$$
and taking into account the standard Stirling formula, we get
$$\begin{array}{lcl}\displaystyle \frac{1}{n}\log
\binom{(rm+1)n+rm+Y}{n-1} & = & \log(rm+1)+rm\log\left(\frac{rm+1}{rm}\right)+o(1)
\enspace,
\end{array}$$
and putting these together, one gets
$$\| P_{\ell}(z)\|_v\leq \exp\left(\dfrac{n[K_v:\Q_w]}{[K:\Q]}\left[\const\right]+o(1)\right)\enspace,$$
where $o(1)\longrightarrow 0$ $(n\rightarrow +\infty)$.

\bigskip

Let $E_0=K[\alpha_i,z]$ the sub-vector space of ${{R}}$ consisting of constants in $t$. Define $$\Theta: E_0\longrightarrow {{R}}; \ Q\mapsto \dfrac{Q(\alpha_i,z)-Q(\alpha_i,t)}{z-t}\enspace.$$
By definition, $P_{\ell,i,s}(z)=\psi_{i,s}\circ \Theta(P_{\ell}(z))$ and $$\psi_{i,s}=[\alpha_i]\circ \ev_{{{t\rightarrow \alpha_i}}}\circ{\mathcal{T}}_{\bold{c}}^{-1}\circ (\theta_t+\zeta_r)\circ \cdots \circ (\theta_t+\zeta_{r-s+1}) \enspace.$$
Using \cite[Lemma $5.2$ $(i), (ii)$]{DHK2} with $N=rn$ and  \cite[Lemma $5.2$ $(iii)$]{DHK2} and Lemma $\ref{fundamental}$  $(iii)$ for $N=rm(n+1)$, and since $rn=\exp(n\cdot o(1))$, one gets $(ii)$.
Finally, we have $$\psi_{i,s}\circ[t^{k+n}](P_{\ell}( t))=[\alpha_i]\circ\ev_{{{t\rightarrow \alpha_i}}} \circ{\mathcal{T}}_{\bold{c}}^{-1}\circ (\theta_t+\zeta_r)\circ \cdots \circ (\theta_t+\zeta_{r-s+1})  \circ
[t^{k+n}](P_{\ell}(t))\enspace.$$ 
Again, using \cite[Lemma $5.2$]{DHK2} and Lemma~\ref{norme}, one gets $(iii)$.
\end{proof}
{{Recall that if $P$ is a homogeneous polynomial in some variables $y_i,i\in I$, for any point $\boldsymbol{\alpha}=(\alpha_i)_{i\in I}\in K^{\card(I)}$ where $I$ is any finite set, and $\Vert\cdot\Vert_v$ stands for the sup norm in $K_v^{\card(I)}$, with $$C_v(P)= (\deg(P)+1)^{\frac{\varepsilon_v{{[K_v:\Q_w]}}(\card(I))}{{{d}}}}\enspace,$$
one has
\begin{equation}\label{estimhomo}
\vert P(\boldsymbol{\alpha})\vert_v\leq C_v(P) \Vert P\Vert_v\cdot {\Vert {\boldsymbol{\alpha}\Vert_v^{\deg(P)}}}\enspace.
\end{equation}
So, the preceding lemma yields trivially estimates for the $v$-adic norm of the above given polynomials.}}
\begin{lemma} \label{upper jyouyonew}
Let $n$ be a positive integer, $\beta\in K$ with $\| \boldsymbol{\alpha}\|_v<\vert \beta\vert_v$. 
Then we have for all $1\le i\leq m,0\le \ell\leq rm, 0 \le s\leq r-1$,
\begin{align*}
|R_{\ell,i,s}(\beta)|_v &\le \| \boldsymbol{\alpha}\|_v^{rm(n+1)}
\cdot \left(\frac{\|\boldsymbol{\alpha}\|_v}{\vert \beta\vert_v}\right)^{n+1}\cdot\left(\frac{\varepsilon_v\vert \beta\vert_v}{\vert \beta\vert_v-\| \boldsymbol{\alpha}\|_v}+(1-\varepsilon_v)\right)\\
&\cdot 
\begin{cases}
\exp\left(n{{\dfrac{[K_v:\Q_w]}{[K:\Q]}}}\left[\const+o(1)\rule{0mm}{4mm}\right]\right) & \text{if} \ v\mid\infty \\
e^{o(1)}|D_{\boldsymbol{c},rmn}\cdot D'_{\boldsymbol{c},rmn}|^{-1}_v \cdot \prod_{j=1}^r \left\vert\mu_{n-1}(\zeta_j)\right\vert_v^{-{{1}}}\rule{0mm}{7mm} & \text{otherwise}\enspace.
\end{cases}
\enspace
\end{align*}
\end{lemma}  
\begin{proof}
By the definition of $P_{\ell}(z)$, as formal power series, we have
$$
R_{\ell,i,s}(z)=\sum_{k=0}^{\infty}\frac{\psi_{i,s}(t^{k+n}P_{\ell}(t))}{z^{k+n+1}} \enspace.
$$
Using the triangle inequality, the fact that $\ell\leq rm$ and Lemma~\ref{majonorme} $(iii)$  and inequality \eqref{estimhomo},
\begin{align*}
\vert R_{\ell,i,s}(\beta)\vert_v&\leq \| \boldsymbol{\alpha}\|_v^{rm(n+1)}\sum_{k=0}^{\infty}\left(\frac{\|\boldsymbol{\alpha}\|_v}{\vert \beta\vert_v}\right)^{n+1+k}\\
&\cdot 
\begin{cases}
\exp\left(n\dfrac{[K_v:\Q_w]}{[K:\Q]}\left[\const+o(1)\rule{0mm}{4mm}\right]\right) & \text{if} \ v\mid\infty \\
e^{o(1)} |D_{\boldsymbol{c},rmn}\cdot D'_{\boldsymbol{c},rmn}|^{-1}_v \cdot \prod_{j=1}^r\left\vert\mu_{n-1}(\zeta_j)\right\vert_v^{-{{1}}}\rule{0mm}{7mm} & \text{otherwise}\enspace,
\end{cases}
\end{align*}
and the lemma follows using geometric series summation.
\end{proof}

\section{Proof of Theorem $\ref{hypergeometric}$}\label{proof}
We use the same notations as in Section $5$.
To prove Theorem $\ref{hypergeometric}$, we shall prove the following theorem.
\begin{theorem} \label{thm 2}
For $v\in {{\mathfrak{M}}}_K$, we define the constants
\begin{align*}
c(x,v)&=\varepsilon_{v}{{\frac{[K_v:\Q_w]}{[K:\Q]}}}\left(\const\right)+(1-\varepsilon_{v})\sum_{j=1}^r\log\vert\mu(\zeta_j)\vert^{-1}_{v} \enspace,
\end{align*}
where $p_v$ is the rational prime under $v$ if $v$ is non-archimedian.
We also define
\begin{align*}
\mathbb{A}_v(\boldsymbol{\eta},\boldsymbol{\zeta},\boldsymbol{\alpha},\beta)&=\log\vert \beta \vert_{v_0}-(rm+1)\log\| \boldsymbol{\alpha}\|_{v_0}-c(x,v_0)+(1-\varepsilon_v)\limsup_{n\to \infty}\dfrac{1}{n}
\log\,|D_{\boldsymbol{c},rmn}\cdot D'_{\boldsymbol{c},rmn}|^{-1}_v\enspace,\\
U_v(\boldsymbol{\eta},\boldsymbol{\zeta},\boldsymbol{\alpha},\beta)&=rm{\mathrm{h}}_{v}(\boldsymbol{\alpha},\beta)+c(x,v)-(1-\varepsilon_{v}) \limsup_{n\to \infty}\dfrac{1}{n}\log\,|D_{\boldsymbol{c},rmn}|_v 
\enspace,
\end{align*}
and
\begin{align*}
V_v(\boldsymbol{\eta},\boldsymbol{\zeta},\boldsymbol{\alpha},\beta) &=\displaystyle  \log\vert\beta\vert_{v_0}-rm{\mathrm{h}}(\boldsymbol{\alpha},\beta)-{{(rm+1)}}\log\| \boldsymbol{\alpha}\|_{v_0}+rm\log\|(\boldsymbol{\alpha},\beta)\|_{v_0}\\
&-\left(\const\right)\\
&-\sum_{j=1}^r\left(\log\mu(\eta_j)+2\log\mu(\zeta_j)+\dfrac{{\rm{den}}(\zeta_j){\rm{den}}(\eta_j)}{\varphi({\rm{den}}(\zeta_j))\varphi({\rm{den}}(\eta_j))} \right)\enspace.
\end{align*}
Let $v_0$ be a place in $\mathfrak{M}_K$,  either archimedean or non-archimedean, such that
$V_{v_0}(\boldsymbol{\eta},\boldsymbol{\zeta},\boldsymbol{\alpha},\beta)>0$.
Then the functions $F_s(\alpha_i/\beta), \,\,0\leq s \leq r-1$ converge around $\alpha_j/\beta$ in $K_{v_0}$,  $1\leq j \leq m$ and
for any positive number $\varepsilon$ with $\varepsilon<V_{v_0}(\boldsymbol{\eta},\boldsymbol{\zeta},\boldsymbol{\alpha},\beta)$, there exists an effectively computable positive number $H_0$ depending on $\varepsilon$ and the given data such that the following property holds.
For any ${{\boldsymbol{\lambda}}}:=({{\lambda_0}},{{\lambda_{i,s}}})_{\substack{1\le i \le m \\ {{0}}\le s \le {{r-1}}}} \in K^{rm+1} \setminus \{ \bold{0} \}$ satisfying $H_0\le {\mathrm{H}}({{\boldsymbol{\lambda}}})$, then we have
\begin{align*}
\left|{{\lambda_0}}+\sum_{i=1}^m\sum_{s=0}^{r-1}{{\lambda_{i,s}}}F_{s}(x,\alpha_i/\beta)\right|_{v_0}>C(\boldsymbol{\eta},\boldsymbol{\zeta},\boldsymbol{\alpha},\beta,\varepsilon){\mathrm{H}}_{v_0}({{\boldsymbol{\lambda}}}) {\mathrm{H}}({{\boldsymbol{\lambda}}})^{-\mu(\boldsymbol{\eta},\boldsymbol{\zeta},\boldsymbol{\alpha},\beta,\varepsilon)}\enspace,
\end{align*}
where 
\begin{align*}
&\mu(\boldsymbol{\eta},\boldsymbol{\zeta},\boldsymbol{\alpha},\beta,\varepsilon):=
\dfrac{\mathbb{A}_{v_0}(\boldsymbol{\eta},\boldsymbol{\zeta},\boldsymbol{\alpha},\beta)+{{U}}_{v_0}(\boldsymbol{\eta},\boldsymbol{\zeta},\boldsymbol{\alpha},\beta)}{V_{v_0}(\boldsymbol{\eta},\boldsymbol{\zeta},\boldsymbol{\alpha},\beta)-\epsilon} \enspace,\\
&C(\boldsymbol{\eta},\boldsymbol{\zeta},\boldsymbol{\alpha},\beta,\varepsilon)=\exp\left(-{{\left(\frac{\log(2)}{V_{v_0}(\boldsymbol{\eta},\boldsymbol{\zeta},\boldsymbol{\alpha},\beta)-\varepsilon}+1\right)}}(\mathbb{A}_{v_0}(\boldsymbol{\eta},\boldsymbol{\zeta},\boldsymbol{\alpha},\beta)+{{U}}_{v_0}(\boldsymbol{\eta},\boldsymbol{\zeta},\boldsymbol{\alpha},\beta)\right)\enspace.
\end{align*}
\end{theorem}
\begin{proof}
By Proposition $\ref{non zero det}$, the matrix ${{\mathrm{M}}}_n=\begin{pmatrix}
P_{\ell}(\beta)\\
P_{\ell,i,s}(\beta)
\end{pmatrix}$ 
with entries in $K$ is invertible. 
By Lemma~\ref{majonorme} $(i)$ together with inequality~(\ref{estimhomo}),
\begin{align*}\displaystyle
\log\|  P_{\ell}(\beta)\|_{v} & \leq \displaystyle \varepsilon_v\left(n{\frac{[K_v:\Q_w]}{[K:\Q]}}\left[\const+o(1)\rule{0mm}{4mm}\right]\right) \\
& +(1-\varepsilon_v)\left[\log\,|D_{\boldsymbol{c},rmn}|^{-1}_v+\sum_{j=1}^r \log\left\vert\mu_n(\zeta_j)\right\vert_v^{-{{1}}}\right] + +(rmn+\ell){\mathrm{h}}_{v}(\boldsymbol{\alpha},\beta)\\
& \leq   \displaystyle n\left(rm{\mathrm{h}}_{v}(\boldsymbol{\alpha},\beta)+c(x,v)\right)+o(n)
\\ & = \displaystyle {{U}}_v(\boldsymbol{\eta},\boldsymbol{\zeta},\boldsymbol{\alpha},\beta)n+o(n)\enspace.
\end{align*}
Similarly, using this time Lemma~\ref{majonorme} $(ii)$ and inequality \eqref{estimhomo}, 
$$
\log\, \|P_{\ell,i,s}(\beta)\|_{v}\leq  \displaystyle n\left(rm{\mathrm{h}}_{v}(\boldsymbol{\alpha},\beta)+c(x,v)\right)+f_v(n)\enspace,
$$
where
\begin{align*}
f_v:\N\rightarrow \R_{\ge0}; \ \ n\longmapsto rm{\mathrm{h}}_{v}(\boldsymbol{\alpha},\beta)+(1-\varepsilon_v)\log\,|D_{\boldsymbol{c},rmn}\cdot D'_{\boldsymbol{c},rmn}|^{-1}_v    \enspace.
\end{align*}
We define $$\displaystyle {{F_v}}(\boldsymbol{\alpha},\beta):\N\rightarrow \R_{\ge0}; \ n\mapsto n\left(rm{\mathrm{h}}_{v}(\boldsymbol{\alpha},\beta)+c(x,v)\right)+f_v(n)\enspace.$$
Since on the other hand, Lemma~\ref{upper jyouyonew} ensures 
{\small{\begin{align*}
-\log\vert  R_{\ell,i,s}( \beta)\vert_{v_0}&
\leq n\log\vert \beta \vert_{v_0}-(rm+1)n \log\| \boldsymbol{\alpha}\|_{v_0}-nc(x,v_0)+(1-\varepsilon_v)\limsup_{n\to \infty}\dfrac{1}{n}\log|D_{\boldsymbol{c},rmn}\cdot D'_{\boldsymbol{c},rmn}|^{-1}_v+o(n)\\
&=\mathbb{A}_{v_0}(\boldsymbol{\eta},\boldsymbol{\zeta},\boldsymbol{\alpha},\beta)n+o(n)\enspace.
\end{align*}}}
Using Lemma $\ref{valuation}$, we have 
\begin{align*}
&{\displaystyle{\lim_{n\to \infty}}}\dfrac{1}{n}{\displaystyle{\sum_{v\in \mathfrak{M}_K}}}f_v(n)=\sum_{j=1}^r\left(\log\,\mu(\eta_j)+\log\,\mu(\zeta_j)+\dfrac{{\rm{den}}(\zeta_j){\rm{den}}(\eta_j)}{\varphi({\rm{den}}(\zeta_j))\varphi({\rm{den}}(\eta_j)})\right)\enspace,\\
&\sum_{v\in \mathfrak{M}_K}c(x,v)=\const+\sum_{j=1}^r{\rm{log}}\,\mu(\zeta_j)\enspace,
\end{align*}
we conclude
\begin{align*}
\mathbb{A}_{v_0}(\boldsymbol{\eta},\boldsymbol{\zeta},\boldsymbol{\alpha},\beta)-\lim_{n\to \infty}\dfrac{1}{n}\sum_{v\neq v_0}{{F_v}}(\boldsymbol{\alpha},\beta)(n)=V(\boldsymbol{\eta},\boldsymbol{\zeta},\boldsymbol{\alpha},\beta)\enspace.
\end{align*}
Applying \cite[Proposition $5.6$]{DHK3} for $\{\theta_{i,s}:=F_{s}(\alpha_i/\beta)\}_{\substack{1\le i \le m \\ 0\le s \le r-1}}$ and the above data, 
we obtain the assertions of Theorem $\ref{thm 2}$. 
\end{proof} 
\begin{proof}[Proof of Theorem $\ref{hypergeometric}$.] 
We use the same notations as in Theorem $\ref{hypergeometric}$ and Theorem $\ref{thm 2}$. 
Put $\boldsymbol{\eta}=(a_1+1,\ldots,a_r+1),\boldsymbol{\zeta}=(b_1,\ldots,b_{r-1},1)$. 
Then we have
$$V_{v_0}(\boldsymbol{\alpha},\beta)=V_{v_0}(\boldsymbol{\eta},\boldsymbol{\zeta},\boldsymbol{\alpha},\beta)\enspace.$$
Combining with \eqref{F0} and \eqref{Fs}, Theorem $\ref{thm 2}$ yields the assertion of Theorem $\ref{hypergeometric}$.
\end{proof}

\bibliography{}

\begin{scriptsize}
\begin{minipage}[t]{0.38\textwidth}

Sinnou David,
\\Institut de Math\'ematiques
\\de Jussieu-Paris Rive Gauche
\\CNRS UMR 7586,
Sorbonne Universit\'{e}
\\
4, place Jussieu, 75005 Paris, France\\
\& CNRS UMI 2000 Relax\\
 Chennai Mathematical Institute\\
 H1, SIPCOT IT Park, Siruseri\\
 Kelambakkam 603103, India \\\\
\end{minipage}
\begin{minipage}[t]{0.32\textwidth}
Noriko Hirata-Kohno, 
\\hirata@math.cst.nihon-u.ac.jp
\\Department of Mathematics
\\College of Science \& Technology
\\Nihon University
\\Kanda, Chiyoda, Tokyo
\\101-8308, Japan\\\\
\end{minipage}
\begin{minipage}[t]{0.3\textwidth}
Makoto Kawashima,
\\kawashima.makoto@nihon-u.ac.jp
\\Department of Liberal Arts \\and Basic Sciences
\\College of Industrial Engineering
\\Nihon University
\\Izumi-chou, Narashino, Chiba
\\275-8575, Japan\\\\
\end{minipage}

\end{scriptsize}

\end{document}